\documentclass[reqno,10pt,twoside]{amsart} 

\usepackage{amsmath,amsthm,amssymb,amsfonts,mathrsfs,color}
\usepackage{latexsym,esint}
\usepackage{mathrsfs}  
\usepackage[colorinlistoftodos,prependcaption,textsize=tiny]{todonotes}
\usepackage{hyperref}
\usepackage{enumitem}

\newtheorem{theorem}{Theorem}[section]
\newtheorem{proposition}{Proposition}[section]

\newtheorem{lemma}{Lemma}[section]

\newtheorem{definition}{Definition}[section]
\newtheorem{remark}{Remark}[section]

\numberwithin{equation}{section} \numberwithin{theorem}{section}
\numberwithin{proposition}{section} \numberwithin{lemma}{section}
\numberwithin{corollary}{section}
\numberwithin{definition}{section} \numberwithin{remark}{section}

\newcommand{\R}{\mathbb R} 

\newcommand{\wto}{\rightharpoonup}

\newcommand{\e}{\varepsilon}

\newcommand{\A}{{\mathcal A}}

\newcommand{\LL}{{\mathcal L}}
\newcommand{\HH}{{\mathcal H}}
\newcommand{\M}{{\mathcal M}}
\newcommand{\C}{{\mathcal C}}

\newcommand{\dd}{\,\mathrm{d}}
\newcommand{\Tr}{{\rm Tr}}

\let\O=\Omega

\newcommand{\dive}{{\rm div }}

\newcommand{\res}{\mathop{\hbox{\vrule height 7pt width .5pt depth 0pt
\vrule height .5pt width 6pt depth 0pt}}\nolimits}

\def \p{\partial}

\author[J.-F. Babadjian]{Jean-Fran\c cois Babadjian}
\author[M. Rakovsky]{ Martin Rakovsky}
\author[R. Rodiac]{R\'emy Rodiac}

\address[J.-F. Babadjian]{Universit\'e Paris-Saclay, CNRS,  Laboratoire de math\'ematiques d'Orsay, 91405, Orsay, France.}
\email{jean-francois.babadjian@universite-paris-saclay.fr}

\address[M. Rakovsky]{Universit\'e Paris-Saclay, CNRS,  Laboratoire de math\'ematiques d'Orsay, 91405, Orsay, France.}
\email{martin.rakovsky@universite-paris-saclay.fr}

\address[R. Rodiac]{Laboratoire J.A. Dieudonn\'e, Universit\'e C\^ote d'Azur, CNRS UNMR 7351,06108, Nice, France.}
\email{remy.rodiac@univcotedazur.fr}
\title[Critical of points of the Ambrosio-Tortorelli functional]{Critical points of the two-dimensional Ambrosio-Tortorelli functional with convergence of the phase-field energy}

\begin{document}

\begin{abstract}
We consider a family \(\{(u_\varepsilon, v_\varepsilon)\}_{\varepsilon>0}\) of critical points of the Ambrosio-Tortorelli functional. Assuming a uniform energy bound, the sequence $\{(u_\varepsilon, v_\varepsilon)\}_{\varepsilon>0}$ converges in $L^2(\Omega)$ to a limit $(u, 1)$ as $\varepsilon \to 0$, where  \(u\) is in \(SBV^2(\Omega)\). It was previously shown that if the full Ambrosio-Tortorelli energy associated to $(u_\e,v_\e)$ converges to the Mumford-Shah energy of $u$, then the first inner variation converges as well. In particular, $u$ is a critical point of the Mumford-Shah functional in the sense of inner variations. In this work, focusing on the two-dimensional setting, we extend this result under the sole convergence of the phase-field energy to the length energy term in the Mumford-Shah functional. 
\end{abstract}

\maketitle

%\tableofcontents

\section{Introduction}

The Mumford-Shah (MS) functional is a prototypical energy used in free-discontinuity problems where the competition between a volume and a surface energy leads to concentration on a co-dimension one discontinuity set. It has been historically introduced in \cite{MS89} in the context of image segmentation, and has also been used  in fracture mechanics in \cite{FM98} to describe the propagation of brittle cracks. 

In order to present this energy, we need to introduce some notation. Let $\Omega$ be a Lipschitz bounded open subset of $\mathbb{R}^N$ with \(N\geq 1\), and  let \(g\in H^{\frac12}(\p \Omega)\) be a Dirichlet boundary data. The energy space, denoted by $SBV^2(\Omega)$, is made of all functions $u \in BV(\Omega)$ with vanishing Cantor part such that the approximate gradient $\nabla u$ belongs to $L^2(\Omega;\R^N)$ and $\mathcal H^{N-1}(J_u)<\infty$, where $J_u \subset \Omega$ is the jump set of $u$ and $\mathcal H^{N-1}$ is the $(N-1)$-dimensional Hausdorff measure. In order to account for possible boundary discontinuities, we introduce the extended jump set \(\widehat{J_u}\)  of \(u\) defined by
\begin{equation}\label{def:bighatJu}
\widehat{J_u}:=J_u \cup (\partial \Omega \cap \{u \neq g \}),
\end{equation}
where, with a slight abuse of notation, we still denote by $u$ the inner trace of $u$ on $\partial \Omega$. We define the Mumford-Shah functional $MS: SBV^2(\Omega) \to \R$ by
\begin{align}\label{def:MS}
MS(u) & =  \int_\Omega |\nabla u|^2\, dx  + \mathcal{H}^{N-1}(J_u)+\mathcal{H}^{N-1}(\partial \Omega \cap \{ u \neq g \})\nonumber\\
 &=\int_\Omega |\nabla u|^2\, dx  + \mathcal{H}^{N-1}(\widehat{J_u}) \quad \text{ for all }u \in SBV^2(\Omega).
\end{align}

Using the direct method in the calculus of variations, it can be proven that this functional admits minimizers as a consequence of Ambrosio's compactness and lower semicontinuity results, see \cite{AFP00,DeGiorgi_Carriero_Leaci_1989,Carriero_Leaci_1990}. The description of the jump set is of particular interest in connexion with the Mumford-Shah conjecture which stipulates that, in dimension $N=2$, the jump set $J_u$ is made of finitely many $\mathcal C^1$ curves intersecting at triple junction type points. In particular, the regularity (or lack of regularity) of the jump set represents a very interesting topic of research also from the application point of view for the description of edges in image segmentation and cracks in mechanics. We refer to the monographs \cite{AFP00,David, DLF} for related results in that direction. 

If one is interested into the numerical approximation of minimizers of the Mumford-Shah energy, we face a difficulty related to the fact that the discontinuity set is not a priori known. It would require a high mesh size precision in the spatial region where $u$ is discontinuous which might lead to too costly numerical schemes.  This is why it is convenient to approximate the MS energy by a more regular functional. In the spirit of the Allen-Cahn model for phase transitions, see e.g.\ \cite{M87,S88}, Ambrosio and Tortorelli proposed in \cite{AT92} the following variational phase-field regularization of the MS energy. In our context, we set
\begin{equation}\label{def:Ag}
\mathcal{A}_g:=\{ (u,v)\in [H^1(\Omega)\times (H^1(\Omega)\cap L^\infty(\Omega))]:\;  u=g \text{ and } v=1 \text{ on }\partial \Omega\}
\end{equation}
and we define the Ambrosio-Tortorelli (AT) energy $AT_\e:\mathcal A_g \to \R$ by 
\begin{equation}\label{eq:defAT}
AT_\varepsilon (u,v) = \int_\Omega (\eta_\varepsilon + v^2) |\nabla u|^2 dx + \int_\Omega \left(\varepsilon |\nabla v|^2 + \frac{(1-v)^2}{4\varepsilon}\right)dx \quad \text{ for all }(u,v)\in \A_g,
\end{equation}
where $0< \eta_\varepsilon \ll \varepsilon$ is a small parameter that ensures ellipticity. This regularization has a mechanical interpretation. The phase-field variable $v$ can be seen as a damage variable taking values in the interval $[0,1]$. The region $\{v=1\}$ corresponds to completely sane material, whereas where $\{v=0\}$ the material is totally damaged. In \cite{AT92}, in a slightly different context (i.e., in absence of Dirichlet boundary conditions), a $\Gamma$-convergence result of $AT_\varepsilon$ to $\widetilde{MS}$ is proved by suitably extending $MS$ as a two variables functional : 
\begin{equation}\label{eq:defMS}
\widetilde{MS}(u,v) = 
\begin{cases}
MS(u) & \text{if } u\in SBV^2(\Omega) \text{ and } v\equiv 1  ,\\
+\infty & \text{otherwise.}
\end{cases}
\end{equation}
 The existence of minimizers in \(\A_g\) for $AT_\varepsilon$ for a fixed value of $\varepsilon$ is obtained via the direct method in the calculus of variations. The fundamental theorem of $\Gamma$-convergence then ensures that a subsequence $\{(u_\varepsilon, v_\varepsilon)\}_{\e>0}$ of minimizers of $AT_\varepsilon$ converges in $[L^2(\Omega)]^2$ to $(u,1)$ where $u$ a minimizer of $MS$. This motivates the use of \(AT_\e\), with \(\e>0\) small, to numerically approximate minimizers of the MS functional. This regularized formulation is indeed at the basis of numerical simulations, see e.g.\ \cite{BFM08}.

However, concerning the numerical implementation, we observe that the term $v^2 |\nabla u|^2$  makes the functional $AT_\varepsilon$ non-convex. Consequently, numerical methods might fail to converge to a minimizer of $AT_\varepsilon$. For example, using the fact that $AT_\varepsilon$ remains separately strictly convex, it was proposed in \cite{BFM08} to perform an alternate minimization algorithm. It is proven in \cite[Theorem 1]{B07} that the sequence of iterates converges to a critical point of $AT_\varepsilon$, but nothing guarantees  this critical point to be a minimizer. It motivates the study of the convergence of critical points of the AT functional as $\varepsilon \rightarrow 0$, and it raises the question to know if their limits correspond to critical points of the MS functional in some sense.

\medskip

We start by recalling that a critical point for the AT functional is a pair $(u_\varepsilon, v_\varepsilon) \in \A_g$ such that 
\begin{equation}\label{eq:critical_AT}
\dfrac{d}{dt}\Bigl|_{t=0}  AT_\varepsilon (u_\varepsilon + t\psi, v_\varepsilon + t\varphi) = 0 \;\;\; \text{ for all } \, (\psi, \varphi) \in \: H_0^1 (\Omega)\times (H_0^1(\Omega) \cap L^\infty(\Omega)).
\end{equation}
Tools of $\Gamma$-convergence provide little help to study the convergence of critical points. The extension of the fundamental theorem of $\Gamma$-convergence to the convergence of critical points has been studied in various settings, see e.g.\ \cite{HT00,T02,T05,B25} for the Allen-Cahn functional and \cite{BBH94,BBO01,LR01,SS07} for the Ginzburg-Landau functional, see also \cite{Braides_2014}. Concerning the AT functional, the general convergence of critical points to critical points of the MS energy is established in dimension $N=1$ in \cite{FLS09,L10,BMRb23}. We also refer to \cite{BI23} for the convergence of critical points of a  phase-field approximation of cohesive fracture energies in 1D.  In dimension $N \geq 2$, the convergence of critical points has been established in \cite{BMRa23}, with the additional assumption of convergence of the energy 
$$AT_\varepsilon (u_\varepsilon, v_\varepsilon) \to MS(u).$$ A similar assumption of convergence of energy was made in \cite{Le_2011,Le_2015,LS19} to study the convergence of critical points of the Allen-Cahn functional to minimal surfaces, and to understand how stability passes to the limit. However this is not always true that convergence of the energy holds. For example, in \cite[Section 6.3]{HT00} it is proved that interfaces with multiplicities can be limits of a sequence of critical points of the Allen-Cahn functional. Consequently these critical points do not satisfy the energy convergence assumption. 

\medskip

In this article we weaken the assumption of convergence of energy made in \cite{BMRa23}. Our main result  Theorem \ref{Theorem 1}, establishes the convergence of critical points under the assumption that only the phase-field term converges, i.e.,
$$ \int_\Omega \left(\varepsilon |\nabla v|^2 + \frac{(1-v)^2}{4\varepsilon}\right)dx \to \mathcal H^{N-1}(\widehat{J_u}).$$
This result is restricted to dimension $N=2$, which is meaningful from the application point of view either in image segmentation, or in the anti-plane brittle fracture model.

\bigskip

Let us close this introduction by some notations and definitions used throughout the paper.

\subsection*{Linear algebra}

Let  \(\mathbb{M}^{N\times N}\) be the set of real \(N \times N\) matrices, and \(\mathbb{M}^{N\times N}_{\text{sym}}\) the set of real symmetric \(N \times N\) matrices. Given two vectors \(a\) and \(b\) in \(\R^N\), we denote by \(a\cdot b \in \R\) their inner product and by $a \otimes b=a^T b \in \mathbb{M}^{N\times N}$ their tensor product. We will use the Frobenius inner product of matrices $A$ and $B \in \mathbb{M}^{N\times N}$, defined by \(A:B= \text{Tr}(A^T B)\) and the associated norm \(|A|=\text{Tr}(A^TA)^{1/2}\). 

\subsection*{Measures} 

The Lebesgue measure in $\R^N$ is denoted by $\LL^N$, and the $k$-dimensional Hausdorff measure by $\HH^k$. 

If $X \subset \R^N$ is a locally compact set and $Y$ an Euclidean space, we denote by $\mathcal M(X;Y)$ the space of $Y$-valued bounded Radon measures in $X$ endowed with the norm $\| \mu \|=|\mu|(X)$, where $|\mu|$ is the variation of the measure $\mu$. If $Y=\R$, we simply write $\mathcal M(X)$ instead of $\mathcal M(X;\R)$. By Riesz representation theorem, $\mathcal M(X;Y)$ can be identified with the topological dual of $\C_0(X;Y)$, the space of continuous functions $f:X \to Y$ such that $\{|f|\geq \e\}$ is compact for all $\e>0$. The weak$^\star$ topology of $\mathcal M(X;Y)$ is defined using this duality.

\subsection*{Functional spaces}

We use standard notation for Lebesgue, Sobolev and H\"older spaces. Given a bounded open set $\O \subset \R^N$, the space of functions of bounded variation is defined by
$$BV(\O)=\{u \in L^1(\O) : \; Du \in \mathcal M(\O;\R^N)\}.$$
We shall also consider the subspace $SBV(\O)$ of special functions of bounded variation made of functions $u \in BV(\O)$ whose distributional derivative can be decomposed as
$$Du=\nabla u \LL^N + (u^+-u^-)\nu_u \HH^{N-1} \res J_u.$$
In the previous expression, $\nabla u$ is the Radon-Nikod\'ym derivative of $Du$ with respect to $\LL^N$, and it is called the approximate gradient of $u$. The Borel set $J_u$ is the (approximate) jump set of $u$. It is a countably $\HH^{N-1}$-rectifiable subset of $\O$ oriented by the (approximate) normal direction of jump $\nu_u :J_u \to \mathbf S^{N-1}$, and $u^\pm$ are the one-sided approximate limits of $u$ on $J_u$ according to $\nu_u$. Finally we define
$$SBV^2(\O)=\{u \in SBV(\O) : \; \nabla u \in L^2(\O;\R^N) \text{ and } \HH^{N-1}(J_u)<\infty\}.$$

We say that a Lebesgue measurable set $E \subset \O$ has finite perimeter in $\O$ if its characteristic function ${\bf 1}_E \in BV(\O)$. We denote by $\partial^\star E$ its reduced boundary. We refer to \cite{AFP00} for a detailed description of the space $BV$.

\subsection*{Varifolds}\label{sec:varifold}

Let us recall several basic ingredients of the theory of varifolds (see e.g. \cite{S83}). We denote by $\mathbf G_{N-1}$ the Grassmannian manifold of all $(N-1)$-dimensional linear subspaces of $\R^N$. The set $\mathbf G_{N-1}$ is as usual identified with the set of all orthogonal projection matrices onto $(N-1)$-dimensional linear subspaces of $\R^N$, i.e., $N \times N$ symmetric matrices $A$ such that $A^2=A$ and ${\rm tr}(A)=N-1$, in other words, matrices of the form 
$$A={\rm Id}-e\otimes e$$
for some $e \in \mathbf S^{N-1}$.

A $(N-1)$-varifold in $X$ (a locally compact subset of $\R^N$) is a bounded Radon measure on $X\times \mathbf G_{N-1}$. The class of $(N-1)$-varifold in $X$ is denoted by $\mathbf V_{N-1}(X)$. The mass of $V\in \mathbf V_{N-1}(X)$ is simply the measure  $\|V\| \in \M(X)$ defined by $\|V\|(B)=V(B \times \mathbf G_{N-1})$ for all Borel sets $B \subset X$.  We define the first variation of an \((N-1)\)-varifold in $V$ in an open set \(U \subset \R^N\) by
$$\delta V(\varphi)=\int_{U \times \mathbf{G}_{N-1}} D \varphi(x):A \dd V(x,A) \quad \text{ for all } \varphi\in \C^1_c(U;\R^N).$$
We say that an \((N-1)\)-varifold is stationary in $U$ if \(\delta V(\varphi)=0\) for all \(\varphi\) in \(\C^1_c(U;\R^N)\). We recall that such a varifold satisfies the monotonicity formula
\begin{equation}\label{eq:monotonicity}
\frac{\|V\|(B_\varrho(x_0))}{\varrho^{N-1}}=\frac{\|V\|(B_r(x_0))}{r^{N-1}}+\int_{(B_\varrho(x_0)\setminus B_r(x_0)) \times \mathbf G_{N-1}}\frac{|P_{A^\perp}(x-x_0)|^2}{|x-x_0|^{N+1}}\dd V(x,A)
\end{equation}
for all $x_0 \in U$ and $0<r<\varrho$ with $B_\varrho(x_0) \subset U$, where $P_{A^\perp}$ is the orthogonal projection onto the one-dimensional space $A^\perp$ (see \cite[paragraph 40]{S83}).

\section{Formulation of the problem and statements of the main result}

Throughout the paper, $\Omega$ is a bounded open subdomain of $\mathbb{R}^2$ with \(\C^{2,1}\) boundary. We denote by $\nu$  the outward unit normal field on $\partial \Omega$.

We recall that the Ambrosio-Tortorelli functional is defined for \( (u,v)\in \mathcal{A}_g\) by \eqref{eq:defAT}, where \(\A_g\) is given by \eqref{def:Ag} and  $\eta_\varepsilon>0$ is a positive small parameter such that $\eta_\varepsilon / \varepsilon \rightarrow 0$. We also recall that the Mumford-Shah functional is defined by \eqref{def:MS}. A critical point \((u_\e,v_\e)\)  of $AT_\varepsilon$ with a prescribed Dirichlet boundary condition is defined in \eqref{eq:critical_AT}. It corresponds to a zero of the outer variations of $AT_\varepsilon$. It is immediate to check that \((u_\e,v_\e) \in \A_g\) is a weak solution of the following system of elliptic partial differential equations:
\begin{equation}\label{criticalpointAT}
\begin{cases}
  -\varepsilon \Delta v_{\varepsilon} + \dfrac{v_{\varepsilon}-1}{4\varepsilon} + v_{\varepsilon} |\nabla u_{\varepsilon}|^2=0 & \text{in } \Omega ,\\
\text{div}\: ((\eta_\varepsilon + v_\varepsilon ^2) \nabla u_\varepsilon) = 0 & \text{in } \Omega,\\
u= g, \quad v=1 & \text{on }\partial \Omega.
\end{cases}
\end{equation}

Whereas the notion of critical points for the AT functional is fairly standard, the notion of critical points for the MS energy is more involved, see \cite[Section 7.4]{AFP00}. We first observe that considering the outer variations of \(MS\) is not sufficient to obtain a proper notion of critical point for $MS$. Indeed, outer variations of $u$ with respect to a smooth direction leaves the surface energy $\mathcal{H}^1(J_u)$ unchanged, so that such variations do not bring any information on the discontinuity set $J_u$. In order to complete the equations derived from the outer variations, we consider \textit{inner variations} of the MS energy. Such variations correspond to deformations of the domain $\Omega$. To define the inner variations of $MS$ up to the boundary, in addition to requiring $\partial \Omega$ to be of class $\mathcal{C}^{2,1}$  we also ask the boundary data \(g\) to belong to the H\"older space \(\C^{2,\alpha}(\p \Omega)\) for some \(\alpha\in (0,1)\). These assumptions are needed to invoke the boundary regularity results of \cite[Theorem 3.2]{BMRa23} which ensures that the weak solutions $(u_\e,v_\e)$ of \eqref{criticalpointAT} are actually classical solutions and belong to \([\C^{2,\alpha}(\overline \Omega)]^2\). Note that, it is necessary, in this type of  variational problems to have some regularity results at our disposal in order to relate inner and outer variations, cf.\ e.g.\ \cite[Section 2]{LS19}. 

\begin{definition}
Let \(X\in \C_c^1(\mathbb{R}^2;\mathbb{R}^2)\) be a vector field satisfying $X\cdot \nu =0$ on $\partial \Omega$. The {\rm flow map} $\Phi_t(x)$ of $X$, defined for every $x$ in $\mathbb{R}^2$, is the unique solution of the Cauchy problem

\[\begin{cases}
\displaystyle \frac{d}{dt}\Phi_t(x)  =  X(\Phi_t(x)) \quad \text{ for all } t \in \R,\\
\displaystyle \Phi_0(x) =  x.
\end{cases}\]
\end{definition}

Thanks to the Cauchy-Lipschitz Theorem, it can be shown that for all $x \in \R^2$, the map \(t \in \R \mapsto \Phi_t(x)\) is globally well-defined, that $(t,x) \mapsto \Phi_t(x)$ belongs to \(\C^1(\R\times \R^2;\R^2)\) and that \(\{\Phi_t\}_{t\in \R}\) is a one-parameter group of \(\C^1\)-diffeomorphisms of \(\R^2\) with \(\Phi_0=\text{Id}\). Furthermore, for each \(t>0\), we can check that \(\Phi_t\) is a \(\C^1\)-diffeomorphism of \(\overline{\Omega}\) which preserves both $\Omega$ and \(\p \Omega\).

\medskip

In the sequel, we consider an arbitrary extension \(G\) of \(g\in \C^{2,\alpha}(\p \Omega)\), which we assume to belong to $\mathcal C^{2,\alpha}_{\text{\rm loc}}(\mathbb{R}^2)$. As already observed in \cite{BMRa23}, the notion of inner variations introduced below a priori depends on the choice of the extension $G$. However, we do not explicitly highlight this dependence in order not to overburden notation.

\begin{definition}
Let  \(X\in \C_c^1(\mathbb{R}^2;\mathbb{R}^2)$ be such that $X\cdot \nu =0$ on $\partial \Omega$ and $\{\Phi_t\}_{t \in \R}$ be  its flow map. For all $u\in SBV^2(\Omega)$, the {\rm first inner variation of $MS$} at $u$ in the direction \(X\) is defined by 
\[\delta MS(u)[X]:= \lim\limits_{t\rightarrow 0}\frac{MS(u\circ \Phi_t^{-1} - G\circ \Phi_t +G) - MS(u)}{t}.\]
\end{definition}

In order to define outer variations, we introduce the following notation. Given $\varphi \in SBV^2(\Omega)$, we set \(\widehat{J_\varphi}=J_\varphi \cup (\partial \Omega \cap \{\varphi \neq g\})\) where we still denote by $\varphi$ the inner trace of $\varphi$ on $\partial \Omega$. In other words, setting $\hat \varphi:=\varphi {\bf 1}_\Omega + G{\bf 1}_{\R^2 \setminus \Omega} \in SBV^2_{\rm loc}(\R^2)$, then $\widehat{J_\varphi}=J_{\hat \varphi}$.

\begin{definition}
Let $\varphi \in SBV^2(\Omega)$ be such that $\widehat{J_\varphi} \subset \widehat{J_u}$. The {\rm first outer variation of $MS$} is defined by 
\begin{equation}\label{eq:growthrate}
dMS(u)[\varphi] :=\lim\limits_{t\rightarrow 0} \frac{MS(u+t\varphi)-MS(u)}{t}.
\end{equation}
\end{definition}

Observe that tests function are taken in $SBV^2(\Omega)$ and not only in $\mathcal{C}_c^\infty(\Omega)$. The assumption $\widehat{J_\varphi} \subset \widehat{J_u}$ ensures the existence of the limit in \eqref{eq:growthrate}. 

\begin{definition}
A function $u \in SBV^2(\Omega)$ is a {\rm critical point} of $MS$ if it satisfies the following two conditions:
$$
\begin{cases}
dMS(u)[\varphi] = 0 & \text{ for all } \varphi \in SBV^2(\Omega) \text{ such that }  \widehat{J_\varphi} \subset \widehat{J_u,}\\
\delta MS(u)[X] =0 & \text{ for all } X \in \mathcal{C}_c^1(\mathbb{R}^2;\mathbb{R}^2) \text{ such that } X\cdot \nu=0 \text{ on } \partial \Omega.
\end{cases}$$
\end{definition}

Computing the inner and the outer variations of $MS$, cf.\ e.g.\ \cite[Appendix A]{BMRa23}, we find that a critical point of $MS$ satisfies 
\begin{equation}\label{eq:3rd_eq_MS}
\int_\Omega \nabla u \cdot \nabla \varphi \, dx=0 \quad \text{ for all } \varphi \in SBV^2(\Omega) \text{ such that } \widehat{J_\varphi} \subset \widehat{J_u}
\end{equation}
and
\begin{multline*}%\label{inner variation of MS}
\int_\Omega (|\nabla u|^2 {\rm Id} - 2\nabla u \otimes \nabla u) : DX  \; dx  + \int_{\widehat{J_u}} ({\rm Id}-\nu_u \otimes \nu_u) : DX \; d\mathcal{H}^1\\
 = -2\int_{\partial \Omega} (\nabla u \cdot \nu) (X\cdot\nabla g) \, d\mathcal{H}^1 \text{ for all } X\in \C_c^1(\R^2;\R^2) \text{ with } X\cdot \nu=0 \text{ on } \p \Omega.
\end{multline*}
Note that if  \(\widehat{J_u}\) is smooth enough, \eqref{eq:3rd_eq_MS} translates into $\Delta u=0$ in $\Omega \setminus J_u$ and \(\p_{\nu_u} u=0\) on \(\widehat{J_u}\) (see \cite[Section 7.4]{AFP00}).

\medskip

One can also define the inner variations of $AT_\varepsilon$.

\begin{definition}\label{def point critique MS}
Let \(X \in \C^1_c(\mathbb{R}^2;\mathbb{R}^2)\) be such that $X\cdot \nu =0$ on $\partial \Omega$ and let $\Phi_t$ be the flow map associated to \(X\). The first inner variation of $AT_\varepsilon$ at $(u , v)$ is defined by
\[\delta AT_\varepsilon (u , v) [X] = \lim\limits_{t \rightarrow 0} \frac{AT_\varepsilon (u \circ \Phi_t^{-1} - G\circ \Phi_t^{-1} + G, v\circ \Phi_t^{-1}) - AT_\varepsilon (u , v)}{t}.\] 
\end{definition}

According to \cite[Theorem 3.1]{BMRa23}, if $(u_\varepsilon , v_\varepsilon)\in \A_g$ is a critical point of $AT_\varepsilon$ in the sense of \eqref{criticalpointAT}, then \( (u_\e,v_\e)\in [\mathcal{C}^\infty( \Omega)]^2\). Moreover, since we assumed that the boundary data \(g\) belongs to \( \C^{2,\alpha}(\p \Omega)\) and $\partial \Omega$ to be of class $\mathcal C^{2,1}$, we deduce from \cite[Theorem 3.2]{BMRa23} that  \( (u_\e,v_\e)\in [\mathcal{C}^{2,\alpha}(\overline \Omega)]^2\). This regularity property ensures that $(u_\e,v_\e)$ is also a zero of the inner variations, i.e., $\delta AT_\varepsilon (u_\varepsilon , v_\varepsilon)=0$, cf.\ e.g.\ \cite[Corollary 2.3]{LS19} or \cite[Proposition 4.2]{BMRa23}. Computing explicitly the inner variations of \(AT_\e\) (see  \cite[Lemma A.3]{BMRa23}), we find that a critical point $(u_\varepsilon , v_\varepsilon)$ of $AT_\varepsilon$ also satisfies 
\begin{multline}\label{variations internes AT}
\int_{\Omega} (\eta_\e+v_\e^2)\left(2 \nabla u_\varepsilon\otimes \nabla u_\varepsilon - |\nabla u_\varepsilon|^2 {\rm Id} \right) : DX \: dx \\
+ \int_\Omega \left[2\e \nabla v_\varepsilon \otimes \nabla v_\varepsilon -\left(\frac{(1-v_\e)^2}{\varepsilon} + \varepsilon |\nabla v_\e|^2\right){\rm Id}\right] : DX \: dx \\
=2(\eta_\varepsilon + 1) \int_{\partial \Omega} (\partial_\nu u_\varepsilon) (X\cdot \nabla g) \, d\mathcal{H}^1
\end{multline}
for all $X \in \C_c^1(\mathbb{R}^2;\R^2)$ such that $X\cdot \nu = 0$ on $\partial \Omega$.

\medskip

We are now in position to state the main result of this work.
\begin{theorem}\label{Theorem 1}
Let $\Omega \subset \R^2$ be a bounded open set with boundary of class $\mathcal C^{2,1}$ and \(g\in \C^{2,\alpha}(\p \Omega)\) for some $\alpha \in (0,1)$. 
Let $\{(u_\varepsilon , v_\varepsilon)\}_{\e>0}$ be a sequence of critical points of $AT_\varepsilon$ in the sense of \eqref{criticalpointAT}. Suppose that the sequence $\{(u_\varepsilon , v_\varepsilon)\}_{\varepsilon>0}$ satisfies the energy bound 
\begin{equation}\label{energy bound}
\sup_{\varepsilon > 0} AT_\varepsilon (u_\varepsilon , v_\varepsilon) < \infty.
\end{equation}
Then, up to extraction, we have that 
\begin{equation}\label{eq:conv_vers_u}
(u_\e,v_\e) \xrightarrow[\e \to 0]{[L^2(\Omega)]^2} (u,1) \text{ with } u \in SBV^2(\Omega) \text{ and } \dive (\nabla u)=0 \text{ in }\mathcal D'(\Omega).
\end{equation}  
If  we further assume the following phase-field energy convergence 
\begin{equation}\label{cv energy}
\int_\Omega \left( \frac{(1-v_\varepsilon)^2}{4\varepsilon} + \varepsilon |\nabla v_\varepsilon|^2\right)dx \rightarrow \mathcal{H}^1(J_u)+\mathcal{H}^1(\partial \O \cap \{u\neq g\})=\mathcal{H}^1(\widehat{J_u}),
\end{equation}
then
\begin{multline}\label{inner variation of MS}
\int_\Omega (|\nabla u|^2 {\rm Id} - 2\nabla u \otimes \nabla u) : DX  \; dx  + \int_{\widehat{J_u}} ({\rm Id}-\nu_u \otimes \nu_u) : DX \; d\mathcal{H}^1\\
 = -2\int_{\partial \Omega} (\nabla u \cdot \nu) (X\cdot\nabla g) \, d\mathcal{H}^1 \text{ for all } X\in \C_c^1(\R^2;\R^2) \text{ with } X\cdot \nu=0 \text{ on } \p \Omega.
\end{multline}
\end{theorem} 

Note that the first point of  \eqref{eq:conv_vers_u} follows from the compactness part of the \(\Gamma\)-convergence theory for \(AT_\e\), see \cite{AT92,BMRa23}. The second point of \eqref{eq:conv_vers_u} is also a rather direct consequence of the \(\Gamma\)-convergence theory for \(AT_\e\). Indeed from this theory we obtain the weak convergence $(\eta_\e +v_\e^2)\nabla u_\e \wto \nabla u$ in $L^2(\Omega;\R^2)$, and we can pass to the limit in the second equation of \eqref{criticalpointAT} to obtain \(\dive (\nabla u)=0\) in $\mathcal D'(\Omega)$. Note that this condition differs from \eqref{eq:3rd_eq_MS} because test functions are smooth, and not discontinuous across $\widehat{J_u}$. Hence, this condition does not suffice to claim that $u$ is a critical point of \(MS\) for the outer variations because the Neumann condition on $\widehat{J_u}$, even in a weak sense, is not satisfied. We leave as an open problem to determine if \eqref{eq:3rd_eq_MS} is satisfied by limits of critical points of the AT energy satisfying the assumptions of Theorem \ref{Theorem 1}, see also Remark \ref{rk_2_Anzellotti}. Finally, the validity of condition \eqref{inner variation of MS} actually shows that $u$ a critical point of $MS$ for the inner variations.

We point out that, whereas the main assumption in \cite[Theorem 1.2]{BMRa23} was 
\begin{equation}\nonumber
AT_\e(u_\e,v_\e) \xrightarrow[\e \to 0]{} MS(u)
\end{equation}
we are able to obtain the same conclusion with the weaker assumption \eqref{cv energy}. However our result is restricted to the two-dimensional case.  We would like to stress that this result does not amount to proving the convergence of the whole energy from the only convergence of the phase-field variable. In fact, we will see in Proposition \ref{prop:sigma} that the elastic term $(\eta_\varepsilon +v_\varepsilon^2)|\nabla u_\varepsilon|^2$ of $AT_\varepsilon$ does not necessarily converge to the elastic term $|\nabla u|^2$ of $MS$ in the sense of measures, and that a singular defect measure might arise. In particular, the lack of strong $L^2(\Omega;\R^2)$-convergence of $\{(\eta_\e+v_\e^2)\nabla u_\e\}_{\e>0}$ to $\nabla u$ is an obstacle to pass to the limit in the second inner variation as in \cite[Theorem 1.3]{BMRa23}.

\medskip

To complete this section, let us explain the strategy of proof of \autoref{Theorem 1} and the organization of the paper. As expected, one attempts to pass to the limit in \eqref{variations internes AT} as $\varepsilon \to 0$. To do so, we first use some consequences of the compactness in the \(\Gamma\)-convergence theory for \(AT_\e\) to $MS$, and of the convergence of the phase-field energy \eqref{cv energy} in Section \ref{sec:prel_res}. They lead to the so-called equi-partition of the (phase-field) energy principle and the identification of the limiting measure of the phase-field density as the one-dimensional Hausdorff measure restricted to the jump set \(\widehat{J_u}\). We also explain how to associate a natural  varifold to the phase-field variable \(v_\e\) following the ideas of  \cite{T02}.  In Section \ref{sec:conv_dirichlet}, we present one of our new ingredients which consists in a finer analysis of the sequence $\{(\eta_\varepsilon + v_\varepsilon^2)\nabla u_\varepsilon \otimes \nabla u_\varepsilon \}_{\varepsilon>0}$. Assumption \eqref{energy bound} implies that this sequence is bounded, so that, up to an extraction, there exists a matrix-valued measure  $\mu$ such that 
\begin{equation}\label{def de mu}
(\eta_\varepsilon + v_\varepsilon^2)\nabla u_\varepsilon \otimes \nabla u_\varepsilon  \mathcal{L}^2 \res \Omega \overset{\star}{\rightharpoonup} \mu \quad \text{ in }\mathcal M(\Omega;\mathbb M^{2 \times 2}_{\rm sym}).
\end{equation}
In Propositions \ref{prop:sigma} and  \ref{cv xor}, we perform a fine analysis of the defect measure $\mu$ by showing that its Lebesgue-absolutely continuous part is given by  \(\nabla u \otimes \nabla u \mathcal{L}^2\res \Omega\). This is achieved by means of a blow-up argument around convenient approximate differentiability points of $u$, which allows us to reduce to the case where the limit function $u$ is actually an affine function. Then the convexity of $AT_\e$ with respect to the phase-field variable enables one to use the equivalence between minimality and criticality with respect to $v$ and then, improve in that specific case the weak $L^2$-convergence of the term $\{\sqrt{\eta_\e+v_\e^2}\nabla u_\e\}_{\e>0}$ to $\nabla u$ into a strong $L^2$-convergence.  We observe that the blow-up argument requires the \(L^2\)-approximate differentiability of \(BV\) functions which holds only in dimension 2. This is one of the reason why our proof is restricted to the dimension $N=2$. Finally Section \ref{sec:proof of Th1} completes the proof of the inner variations convergence thanks to a result of De Phillipis-Rindler \cite{DePhilippis_Rindler_2016} on the singular part of measures satisfying a linear PDE. To this aim, we introduce the stress-energy tensor
$$T_\e:=(\eta_\varepsilon + v_\varepsilon^2)(2 \nabla u_\varepsilon\otimes \nabla u_\varepsilon - |\nabla u_\varepsilon|^2 {\rm Id}) + 2\e\nabla v_\varepsilon \otimes \nabla v_\varepsilon -\left(\frac{(1-v_\e)^2}{\varepsilon} + \varepsilon |\nabla v_\e|^2\right){\rm Id}.$$
According to \eqref{variations internes AT}, $T_\e$  is a divergence measure matrix-valued function in $\Omega$. The energy bound \eqref{energy bound} ensures that $T_\e \overset{\star}{\rightharpoonup} T$ weakly$^\star$ as measures in $\Omega$, for some $T \in \mathcal M(\Omega;\mathbb M^{2 \times 2}_{\rm sym})$ which remains divergence free. Applying the result of  \cite{DePhilippis_Rindler_2016} implies that the singular part of $T$ has a polar taking values into the set of singular matrices. In particular, it ensures that the singular part of $T$ with respect to the Lebesgue measure must be absolutely continuous with respect to  \(\mathcal{H}^1 \res \widehat{J_u}\). The density of $T$ with respect to $\HH^1\res\widehat{J_u}$ is then characterized by means of a blow-up argument. We then employ the previous information on the absolutely continuous part of \(\mu\) with respect to the Lebesgue measure to deduce that the limit map \(u\) in Theorem \ref{Theorem 1} does satisfy \eqref{inner variation of MS}.

\section{Preliminary results}\label{sec:prel_res}

We first show that the energy bound assumption \eqref{energy bound} provides compactness properties leading to \eqref{eq:conv_vers_u}.

\begin{lemma}\label{cv faible vgradu}
Let $\{(u_\varepsilon , v_\varepsilon)\}_{\e>0} \subset \A_g$ be a sequence of critical points of $AT_\e$ satisfying \eqref{energy bound}. Then, up to a subsequence (not relabeled), there exists $u \in SBV^2(\Omega)$ such that $(u_\e,v_\e) \to (u,1)$ in $[L^2(\Omega)]^2$, 
\begin{equation}
\sqrt{v_\varepsilon ^2 +\eta_\varepsilon} \nabla u_\varepsilon \rightharpoonup \nabla u \quad \text{in } \, L^2(\Omega;\R^2).
\end{equation}
and 
$${\rm div}(\nabla u)=0 \quad \text{ in }\mathcal D'(\Omega).$$
\end{lemma}

\begin{proof}
First of all, the energy bound immediately implies that $v_\e \to 1$ in $L^2(\Omega)$.

Next according to the maximum principle (see \cite[Lemma 3.1]{BMRa23}), we have $0 \leq v_\e \leq 1$ in $\Omega$, and $\|u_\e\|_{L^\infty(\Omega)} \leq \|g\|_{L^\infty(\partial \Omega)}$. In particular, the sequence $\{u_\e\}_{\e>0}$ is bounded in $L^\infty(\Omega)$.

Let $0<a<b<1$, we use Young's inequality and the coarea formula to get that
\begin{align}
AT_\e(u_\e,v_\e) &\geq \int_\Omega v_\e^2 |\nabla u_\e|^2\, dx + \int_\Omega (1-v_\e)|\nabla v_\e|\, dx\\
& \geq \int_\Omega v_\e^2 |\nabla u_\e|^2\, dx+\int_a^b (1-t)\HH^{1}(\partial^\star \{v_\e>t\})\, dt.
\end{align}
By the mean value Theorem, there exists $t_\e \in (a,b)$ such that the set $A_\e=\{v_\e>t_\e\}$ has finite perimeter in $\Omega$ and
$$AT_\e(u_\e,v_\e) \geq t_\e^2 \int_{A_\e} |\nabla u_\e|^2 \, dx + \left[\frac{(1-a)^2}{2}-\frac{(1-b)^2}{2}\right]  \HH^1(\partial^\star A_\e).$$

Let $\hat u_\e:=u_\e {\bf 1}_{A_\e}$. According to \cite[Theorem 3.84]{AFP00}, $\hat u_\e \in SBV^2(\Omega)$ with $\nabla \hat u_\e=\nabla u_\e {\bf 1}_{A_\e}$ and $J_{\hat u_\e} \subset \partial^\star A_\e$ so that, using $t_\e>a$,
$$\int_\Omega |\nabla \hat u_\e|^2\, dx + \HH^1(J_{\hat u_\e}) \leq C_{a,b},$$
for some constant $C_{a,b}>0$. Moreover, recalling that $\|u_\e\|_{L^\infty(\Omega)} \leq \|g\|_{L^\infty(\partial \Omega)}$, it follows that the sequence $\{\hat u_\e\}_{\e>0}$ is also bounded in $L^\infty(\Omega)$. We are thus in position to apply Ambrosio's compactness Theorem (Theorem 4.8 in \cite{AFP00}) which ensures, up to the extraction of a subsequence (not relabelled), the existence of $u \in SBV^2(\Omega) \cap L^\infty(\Omega)$ such that $\hat u_\e \to u$ in $L^2(\Omega)$, $\|u\|_{L^\infty(\Omega)} \leq \|g\|_{L^\infty(\partial \Omega)}$ and $\nabla \hat u_\e \wto \nabla u$ in $L^2(\Omega;\R^2)$.

According to the Chebychev inequality, we have
$$\LL^2(\Omega \setminus A_\e) \leq \frac{1}{(1-t_\e)^2}\int_\Omega (1-v_\e)^2\, dx \leq \frac{1}{(1-a)^2}\int_\Omega (1-v_\e)^2\, dx \to 0,$$
where we used the energy bound \eqref{energy bound}. As a consequence,
$$\|u_\e-u\|_{L^2(\Omega)} \leq \|u_\e-\hat u_\e\|_{L^2(\Omega)} + \|\hat u_\e-u\|_{L^2(\Omega)} \leq 2\|g\|_{L^\infty(\partial\Omega)} \LL^2(\Omega \setminus A_\e)+ \|\hat u_\e-u\|_{L^2(\Omega)}\to 0,$$
hence $u_\e \to u$ in $L^2(\Omega)$. Similarly, for all $\varphi \in \mathcal C^\infty_c(\Omega;\R^2)$, we have
\begin{multline}\label{eq1}
\int_\Omega (\sqrt{\eta_\e+v_\e^2} \nabla u_\e - \nabla u) \cdot \varphi\, dx  = \int_\Omega (\nabla \hat u_\e - \nabla u) \cdot \varphi\, dx\\
+ \int_{\Omega \setminus A_\e} \sqrt{\eta_\e+v_\e^2} \nabla u_\e \cdot \varphi\, dx+ \int_{A_\e} (\sqrt{\eta_\e+v_\e^2} \nabla u_\e - \nabla u_\e) \cdot \varphi\, dx.
\end{multline}
On the one hand, we have 
\begin{equation}\label{eq2}
\int_\Omega (\nabla \hat u_\e - \nabla u) \cdot \varphi\, dx \to 0.
\end{equation}
On the other hand, thanks to the Cauchy-Schwarz inequality,
\begin{equation}\label{eq3}
\left| \int_{\Omega\setminus A_\e} \sqrt{\eta_\e+v_\e^2} \nabla u_\e \cdot \varphi\, dx\right| \leq \|\varphi\|_{L^\infty(\Omega)} \left\|\sqrt{\eta_\e+v_\e^2}\nabla u_\e\right\|_{L^2(\Omega)} \LL^2(\Omega \setminus A_\e)^{1/2} \to 0,
\end{equation}
where we used once more the energy bound. Finally, since $v_\e >t_\e>a>0$ in $A_\e$ and $\{\sqrt{\eta_\e+v_\e^2}\nabla u_\e\}_{\e>0}$ is bounded in $L^2(\Omega;\R^2)$, we infer that $\{\|\nabla u_\e\|_{L^2(A_\e)}\}_{\e>0}$ is bounded. Thus, using again the Cauchy-Schwarz inequality and the energy bound, we get that
\begin{equation}\label{eq4}
\left| \int_{A_\e} (\sqrt{\eta_\e+v_\e^2} \nabla u_\e - \nabla u_\e) \cdot \varphi\, dx\right|\leq \|\varphi\|_{L^\infty(\Omega)} \|\nabla u_\e\|_{L^2(A_\e)} \left\|\sqrt{\eta_\e+v_\e^2}-1\right\|_{L^2(\Omega)} \to 0.
\end{equation}
Gathering \eqref{eq1}--\eqref{eq4} together with the boundedness of $\{\sqrt{\eta_\e+v_\e^2} \nabla u_\e\}_{\e>0}$ in $L^2(\Omega;\R^2)$, we conclude that $\sqrt{\eta_\e+v_\e^2} \nabla u_\e \wto \nabla u$ weakly in $L^2(\Omega;\R^2)$.

\medskip

Since $\sqrt{\eta_\e+v_\e^2} \to 1$ in $L^2(\Omega)$ and ${\rm div}((\eta_\e+v_\e^2)\nabla u_\e)=0$ in $\mathcal D'(\Omega)$, we deduce that ${\rm div}(\nabla u)=0$ in $\mathcal D'(\Omega)$.
\end{proof}

The phase-field energy convergence assumption \eqref{cv energy} has several consequences, the main one being the so called equi-partition of the energy (see \cite[Proposition 5.1]{BMRa23}).

\begin{proposition}\label{equipartition energy}
Let $\{(u_\varepsilon , v_\varepsilon)\}_{\e>0}$ be a sequence of critical points of $AT_\varepsilon$ satisfying the assumptions of Theorem \ref{Theorem 1}. We define the discrepancy 
$$\xi_\varepsilon = \varepsilon |\nabla v_\varepsilon|^2 - \frac{(1-v_\varepsilon)^2}{4\varepsilon}.$$ 
Then $\xi_\varepsilon \rightarrow 0$ in $L^1(\Omega)$ and setting
\begin{equation}\label{def we}
w_\varepsilon = \Phi(v_\varepsilon) \;\;\; \text{with} \; \Phi(t) = t-t^2/2,
\end{equation} 
then
\[|\nabla w_\varepsilon | \mathcal{L}^2\res \Omega \overset{\star}{\rightharpoonup} \mathcal{H}^1 \res \widehat{J_u} \;\;\; \text{weakly}^\star \text{ in } \mathcal{M}(\overline{\Omega}).\]
\end{proposition}

The equi-partition of the energy is crucial in \cite{BMRa23} to pass to the limit in \eqref{variations internes AT}. Combining \cite[Proposition 1.80]{AFP00} and  \eqref{cv energy}, we find the following convergence.

\begin{proposition}\label{cvmesureÃ©nergie}
Let $\{(u_\varepsilon , v_\varepsilon)\}_{\e>0}$ be a sequence of critical points of $AT_\varepsilon$ satisfying the assumptions of Theorem \ref{Theorem 1}. Then
\[\left(\frac{(1-v_\varepsilon)^2}{4\varepsilon} + \varepsilon |\nabla v_\varepsilon|^2 \right) \mathcal{L}^2 \res \Omega \overset{\star}{\rightharpoonup} \mathcal{H}^1 \res \widehat{J_u} \;\;\;  \text{weakly}^\star \text{ in } \mathcal{M}(\overline{\Omega}).\]
\end{proposition}

The energy bound also implies  a weak-convergence result for the boundary terms, which corresponds to \cite[Lemma 4.1]{BMRa23}.

\begin{lemma}\label{boundary term BMR}
Let $\{(u_\varepsilon , v_\varepsilon)\}_{\e>0}$ be a sequence of critical points of $AT_\varepsilon$ satisfying the assumptions of Theorem \ref{Theorem 1}. Up to a further subsequence (not relabeled), $\partial_\nu u_\varepsilon \rightharpoonup \nabla u \cdot \nu$ in $L^2(\partial \Omega)$. Moreover, there exists a non-negative boundary Radon measure $m \in \mathcal{M}(\partial \Omega)$ such that

\[\left[|\partial_\nu u_\varepsilon|^2 + \varepsilon |\partial_\nu v_\varepsilon|^2 \right]\mathcal{H}^1 \res \partial \Omega \overset{\star}{\rightharpoonup} m \quad \text{weakly}^\star \text{ in }\mathcal{M}(\partial\Omega).\]
\end{lemma}

The results concerning the measure $\mu$ and the passage to the limit in the inner variations of $AT_\varepsilon$ find a better formulation in terms of varifolds. We now recall the expression of the varifold associated to the phase-field function $v_\e$, and some of the convergence results established in \cite{BMRa23}. 

\begin{definition}\label{definition varifold}
Let $\{(u_\varepsilon , v_\varepsilon)\}_{\e>0}$ be a sequence satisfying \(\sup_{\e} AT_\e(u_\e,v_\e) <+\infty\) and let $w_\e$ be defined by \eqref{def we}. We introduce the varifold $V_\varepsilon\in \mathcal M(\overline{\Omega}\times \mathbf G_1)$  by 

\begin{equation}\label{def varifold}
\langle V_\varepsilon , \varphi \rangle = \int_{\Omega \cap \{\nabla w_\varepsilon \neq 0\}} \varphi \left(x , {\rm Id} - \frac{\nabla w_\varepsilon}{|\nabla w_\varepsilon|} \otimes \frac{\nabla w_\varepsilon}{|\nabla w_\varepsilon|} \right) |\nabla w_\varepsilon| \, dx \;\;\;\;\; \text{for all} \; \varphi \in \mathcal{C}(\overline{\Omega}\times \mathbf G_1).
\end{equation}
\end{definition}

Definition \ref{def varifold} implies that the weight measure of $V_\e$ is given by $\|V_\varepsilon\| = |\nabla w_\varepsilon | \mathcal{L}^2\res \Omega$. In \cite[Section 5]{BMRa23} the following convergence result of Proposition \ref{desintegration} is established. It is an immediate consequence of the energy bound \eqref{energy bound}, the weak$^\star$ compactness of bounded sequences of Radon measures and the Disintegration Theorem (see e.g.\ \cite[Theorem 2.28]{AFP00})

\begin{proposition}\label{desintegration}
Let $\{(u_\varepsilon , v_\varepsilon)\}_{\e>0}$ be a sequence satisfying \(\sup_{\e} AT_\e(u_\e,v_\e) <+\infty\), and let $V_\e$ be the varifold defined in \eqref{def varifold}.  Up to a subsequence (not relabeled), there exists a varifold $V \in \mathcal{M}(\overline{\Omega}\times \mathbf G_1)$ such that
\begin{equation}
\label{cv varifold}
V_\varepsilon \overset{\star}{\rightharpoonup} V \;\;\; \text{weakly}^\star \text{ in }\, \mathcal{M}(\overline{\Omega}\times \mathbf G_1).
\end{equation}
Moreover, there exists a weak$^\star$ $(\mathcal{H}^1 \res \widehat{J_u})$-measurable mapping $x \mapsto V_x \in \mathcal M(\mathbf G_1)$ of probability measures such that $V = (\mathcal{H}^1 \res \widehat{J_u}) \otimes V_x$, i.e.
$$\int_{\overline \Omega \times \mathbf G_1} \varphi(x,A)\, dV(x,A)=\int_{\widehat{J_u}} \left(\int_{\mathbf G_1} \varphi(x,A)\, dV_x(A)\right)d\mathcal H^1(x) \quad \text{ for all } \varphi \in \mathcal{C}(\overline{\Omega}\times \mathbf G_1).$$
\end{proposition}

Using the disintegrated structure of $V$, we have an alternative expression of the first variation of the varifold \(V\).

\begin{proposition}\label{premier moment}
For $\mathcal{H}^1$-a.e $x \in \widehat{J_u}$, we define the first moment of $V_x$ as 
\begin{equation}\label{first moment}
\overline A(x) := \int_{\mathbf G_{1}} A \, dV_x(A). 
\end{equation}
Then, the matrix $\overline A(x) \in \mathbb M^{2 \times 2}_{\rm sym}$ is an orthogonal projector for $\mathcal{H}^1$-a.e $x$ in $\widehat{J_u}$ and\footnote{the first variation of a varifold is usually defined on an open set. We extend this definition to treat the boundary here.}
\begin{equation}\label{eq:first_var_barycentre}
\delta V (X) = \int_{\widehat{J_u}} \overline A(x) : DX(x) \, d\mathcal{H}^1(x) \quad \text{ for all } X\in \C^1 (\overline{\Omega};\R^2).
\end{equation}
Moreover, for $\mathcal{H}^1$-a.e $x\in \widehat{J_u}$, we have \(V_x=\delta_{\overline A(x)}\). 
\end{proposition}

\begin{proof}
According to \cite[Lemma 5.2]{BMRa23}, for $\mathcal{H}^1$-a.e $x \in \widehat{J_u}$, the matrix $\overline A(x) \in \mathbb M^{2 \times 2}_{\rm sym}$ satisfies $\Tr(\overline A(x))=1$ and $\rho(\overline A(x))=1$, where \(\rho\) denotes the spectral radius. In dimension two, those conditions imply that the eigenvalues of $\overline A(x)$ are exactly $0$ and $1$, which ensures that $\overline A(x)$ is an orthogonal projector.

Concerning the representation of the first variation of $V$, we consider a test function $X \in \mathcal{C}^1(\overline{\Omega};\R^2)$. According to Proposition \ref{desintegration} with $\varphi(x,A) = DX(x):A$ together with Fubini's Theorem, we infer that
\begin{align*}
\delta V(X)=\int_{\overline{\Omega} \times \mathbf G_{1}} DX(x):A \, dV(x,A)&=\displaystyle\int_{\widehat{J_u}} \left(\int_{\mathbf G_{1}} DX(x):A\,  dV_x(A) \right) d \mathcal{H}^1(x)\\
& =\displaystyle\int_{\widehat{J_u}} \left(\int_{\mathbf G_{1}} A \, dV_x(A) \right) :DX(x)\, d \mathcal{H}^1(x),
\end{align*}
which corresponds to \eqref{eq:first_var_barycentre} by definition \eqref{first moment} of $\overline A$.

Since any element $A \in \mathbf G_1$ satisfies $|A|=1$ and $V_x$ is a probability measure over $\mathbf G_1$ for $\HH^1$-a.e. $x \in \widehat{J_u}$, we infer that the variance of $V_x$ is zero, i.e.
\[\int_{\mathbf G_1} | \overline A(x) -  A|^2 \, dV_x(A) = 0,\]
and thus, $A=\overline A(x)$ for $V_x$-a.e $A\in \mathbf G_1$. As a consequence, $V_x$ is a probability measure concentrated at $\{\overline A(x)\}$, hence the Dirac measure at $\overline A(x)$. 
\end{proof}

In higher dimension, Proposition \ref{premier moment} remains true under the hypothesis \eqref{cv energy}. A direct consequence is that the probability measure $V_x$ is actually a Dirac mass concentrated at $\overline A(x)$.

%\begin{remark}{\rm
%An alternative proof is the following. Let \(x\) in \(\widehat{J_u}\) be such that \(\overline A(x)\) is a projector. Since the elements of $\mathbf G_1$ are projectors as well, for every $n$ in \(\mathbb{N}\) one has : 
%\[\overline A(x)^n = \overline A(x) = \int_{\mathbf G_1} A \, dV_x(A) = \int_{\mathbf G_1} A^n \, dV_x(A).\]
%Hence, for every real polynomial function $P:\mathbb M^{2 \times 2} \to \mathbb M^{2 \times 2}$, one has 
%\[P(\overline A(x)) = \int_{\mathbf G_1} P(A)\,  dV_x(A).\]
%The Weierstrass Approximation Theorem and the compactness of $\mathbf G_1$ allow us to extend this formula to every function $P \in \mathcal{C}(G_1)$ and we obtain the conclusion.   }
%\end{remark}

\section{Analysis of the defect measure}\label{sec:conv_dirichlet}

The aim of this section is to prove the following result, which is independent of the phase-field energy convergence assumption.
 
\begin{proposition}\label{prop:sigma}
Let $\{(u_\e,v_\e)\}_{\e>0}$ be family of  critical points of the Ambrosio-Tortorelli functional such that 
$$\sup_{\e>0} AT_\e(u_\e,v_\e)<\infty.$$
Then, there exist a subsequence (not relabeled) and a non-negative measure $\sigma \in \mathcal M(\Omega)$ which is singular with respect to the Lebesgue measure, such that
\begin{equation}\label{cv dirichlet}
(\eta_\e + v_\varepsilon^2) |\nabla u_\varepsilon|^2\LL^2\res\Omega  \overset{\star}{\rightharpoonup} |\nabla u|^2\LL^2 \res \Omega+ \sigma\quad\text{weakly}^\star \text{ in }\mathcal M(\Omega).
\end{equation}
\end{proposition}

\begin{proof}
Since the families $\{(\eta_\e + v_\varepsilon^2) |\nabla u_\varepsilon|^2\}_{\e>0}$ and $\{\varepsilon |\nabla v_\varepsilon|^2 + (1-v_\varepsilon)^2/4\varepsilon\}_{\varepsilon >0}$ are bounded in $L^1(\Omega)$ by the energy bound \eqref{energy bound}, up to a subsequence (not relabelled), there exist nonnegative measures $\lambda_1$ and $\lambda_2\in \mathcal M(\Omega)$ such that 
$$(\eta_\e + v_\varepsilon^2) |\nabla u_\varepsilon|^2\LL^2  \res\Omega \overset{\star}{\rightharpoonup} \lambda_1\quad \text{ weakly$^\star$ in } \mathcal M(\Omega)$$
and
$$\left(\e|\nabla v_\e|^2 + \frac{(v_\e-1)^2}{4\e}\right)\LL^2\res\Omega \overset{\star}{\rightharpoonup} \lambda_2\quad \text{ weakly$^\star$ in } \mathcal M(\Omega).$$
On the other hand, since from Lemma \ref{cv faible vgradu} we know that $\sqrt{\eta_\e + v_\varepsilon^2}\nabla u_\varepsilon \rightharpoonup \nabla u$ weakly in $L^2(\Omega;\R^2)$, it results from the weak lower semi-continuity of the norm that $\lambda_1 \geqslant |\nabla u|^2\LL^2\res \Omega$. Thus the measure $\sigma=\lambda_1- |\nabla u|^2\mathcal L^2\res \Omega$ is nonnegative, and it remains to show that it is singular with respect to $\LL^2$. To this aim, by Lebesgue's differentiation Theorem, it is enough to prove that
\begin{equation}\label{eq:sigma-singular}
\frac{d\lambda_1}{d\LL^2}=|\nabla u|^2 \quad \LL^2\text{-a.e. in }\Omega.
\end{equation}

We write \( \lambda_1=\frac{d\lambda_1}{d\LL^2} \mathcal{L}^2 +\lambda_1^s\) with \(\lambda_1^s\) which is singular with respect to the Lebesgue measure. Let $x_0 \in \Omega$ be 
\begin{itemize}
\item[(i)] such that $\frac{d\lambda_1}{d\LL^2}(x_0)$ and $\frac{d\lambda_2}{d\LL^2}(x_0)$ exist and are finite;
\item[(ii)] a Lebesgue point of $u$, $\nabla u$ and $\frac{d\lambda_1}{d\LL^2}$;
\item[(iii)] such that 
\begin{equation}\label{condition iii}
\lim\limits_{\rho\rightarrow 0} \frac{1}{\pi \rho^2} \int_{B_\rho(x_0)} \frac{|u(y) -u(x_0)- \nabla u(x_0)\cdot (y-x_0)|^2}{\rho^2}\, dy = 0;
\end{equation}
\item[(iv)] \label{condition iv} such that $\frac{d\lambda_1^s}{d\LL^2}(x_0)=0$.
\end{itemize}
Note that $\LL^2$ almost every points $x_0$ in $\Omega$ fullfil these properties by the Lebesgue differentiation Theorem, the \(L^2\)-differentiability a.e.\ of \(BV\) functions cf.\ \cite[Theorem 1, Section 6.1]{EG92} and the fact that $\lambda_1^s$ is singular with respect to $\LL^2$. We point out that the \(L^2\)-differentiability a.e.\ of \(BV\) functions holds only in dimension 2 which is one of the reasons of our dimensional restriction.

\medskip

Let $R>0$ be such that $B_R(x_0) \subset\subset \Omega$, and $\{\varrho_j\}_{j \in \mathbb N} \searrow 0^+$ be an infinitesimal sequence of radii such that $\rho_j < R$ for all $j \in \mathbb N$ and
\begin{equation}\label{eq:negligible_circles}
\lambda_1(\partial B_{\rho_j}(x_0))+\lambda_2(\partial B_{\rho_j}(x_0))=0.
\end{equation}

\noindent For all $y$ in $B_1$, we define the rescaled functions.
$$\hat v_{\e,j}(y)=v_\e(x_0+\rho_j y), \quad \hat u_{\e,j}(y)=\frac{u_\e(x_0+\rho_j y)-u(x_0)}{\rho_j}.$$
Then, since $(u_\e,v_\e) \to (u,1)$ in $[L^2(\Omega)]^2$,
$$\lim_{j \to \infty} \lim_{\e \to 0} \int_{B_1}|\hat v_{\e,j}-1|^2\, dy=\lim_{j \to \infty} \lim_{\e \to 0}\frac{1}{\rho_j^2}\int_{B_{\rho_j}(x_0)}|v_\e-1|^2\, dx=0,$$
while
\begin{multline*}
\lim_{j \to \infty} \lim_{\e \to 0}\int_{B_1}|\hat u_{\e,j}(y)-\nabla u(x_0)\cdot y|^2\, dy\\
=\lim_{j \to \infty} \int_{B_1}\left|\frac{u(x_0+\rho_j y)- u(x_0)}{\rho_j}-\nabla u(x_0)\cdot y\right|^2dy\\
= \lim_{j \to \infty}\frac{1}{\rho_j^2} \int_{B_{\rho_j}(x_0)}\frac{|u(x)- u(x_0)-\nabla u(x_0)\cdot (x-x_0)|^2}{\rho_j^2}\, dx=0,
\end{multline*}
where we used \eqref{condition iii}. Moreover, for all $\zeta \in \mathcal C_c(B_1)$, 
\begin{multline*}
\lim_{j \to \infty} \lim_{\e \to 0}\int_{B_1} ( \eta_\e+\hat v_{\e,j}^2)|\nabla \hat u_{\e,j}|^2\zeta\, dy\\
=\lim_{j \to \infty} \lim_{\e \to 0}\frac{1}{\rho_j^2} \int_{B_{\rho_j}(x_0)} (\eta_\varepsilon + v_\e(x)^2)|\nabla u_\e(x)|^2 \zeta\left(\frac{x-x_0}{\rho_j}\right)dx\\
=\lim_{j \to \infty}\frac{1}{\rho_j^2} \int_{B_{\rho_j}(x_0)}  \zeta\left(\frac{x-x_0}{\rho_j}\right)d\lambda_1(x).
\end{multline*}
Condition (iv) implies that 
\[\frac{1}{\rho_j^2} \left|\int_{B_{\rho_j}(x_0)}  \zeta\left(\frac{x-x_0}{\rho_j}\right) d \lambda_1^s \right| \leqslant \|\zeta\|_{L^\infty(B_1)} \frac{\lambda_1^s(B_{\rho_j}(x_0))}{\rho_j^2} \underset{j\rightarrow +\infty}{\longrightarrow } 0.\]
Hence, using item (ii),
\begin{align*}
\lim_{j \to \infty}\frac{1}{\rho_j^2} \int_{B_{\rho_j}(x_0)}  \zeta\left(\frac{x-x_0}{\rho_j}\right)d\lambda_1(x)& =\lim_{j \to \infty}\frac{1}{\rho_j^2} \int_{B_{\rho_j}(x_0)}  \zeta\left(\frac{x-x_0}{\rho_j}\right) \frac{d\lambda_1}{d\LL^2}(x)\, dx\\
&= \frac{d\lambda_1}{d\LL^2}(x_0)\int_{B_1}\zeta\, dy.
\end{align*}
We next observe that, thanks to \eqref{eq:negligible_circles}, (i) and (iv),
$$\lim_{j \to \infty} \lim_{\e \to 0} \frac{1}{\pi \rho_j^2} \int_{B_{\rho_j}(x_0)} (\eta_\varepsilon + v_\e^2)|\nabla u_\e|^2\, dx =\lim_{j \to \infty}  \frac{\lambda_1(B_{\rho_j}(x_0))}{\pi \rho_j^2} =\frac{d\lambda_1}{d\LL^2}(x_0)$$
and
$$\lim_{j \to \infty} \lim_{\e \to 0} \frac{1}{\pi \rho_j^2} \int_{B_{\rho_j}(x_0)}\left(\e|\nabla v_\e|^2 + \frac{(v_\e-1)^2}{4\e}\right)dx
=\lim_{j \to \infty}  \frac{\lambda_2(B_{\rho_j}(x_0))}{\pi \rho_j^2}=\frac{d\lambda_2}{d\LL^2}(x_0),$$
while, changing variables
$$\frac{1}{\rho_j^2} \int_{B_{\rho_j}(x_0)} (\eta_\varepsilon + v_\e^2)|\nabla u_\e|^2\, dx =\int_{B_1} (\eta_\e+\hat v_{\e,j}^2)|\nabla \hat u_{\e,j}|^2\, dy,$$
and
$$\frac{1}{\rho_j^2} \int_{B_{\rho_j}(x_0)} \left( \e|\nabla v_\e|^2 + \frac{(v_\e-1)^2}{4\e}\right)dx
=\frac{1}{\rho_j} \int_{B_1} \left((\e/\rho_j) |\nabla \hat v_{\e,j}|^2 + \frac{(\hat v_{\e,j}-1)^2}{4(\e/\rho_j)}\right)dy.$$

Using a diagonal extraction argument together with the separability of $\mathcal \C_c(B_1)$, we can find an infinitesimal sequence $\e_j \to 0$ such that, setting $\hat \e_j:=\e_j/\rho_j$, $\hat v_j:=\hat v_{\e_j,j}$, $\hat u_j:=\hat u_{\e_j,j}$ and $\hat u(y)=\nabla u(x_0)\cdot y$, then 
\begin{equation}\label{eq:convergences}
\begin{cases}
\hat \e_j \to 0,\\
(\hat u_j,\hat v_j) \to (\hat u,1) \text{ in }[L^2(B_1)]^2,\\
(\eta_{\e_j}+\hat v_{j}^2)|\nabla \hat u_{j}|^2\LL^2 \res\Omega \rightharpoonup   \frac{d\lambda_1}{d\LL^2}(x_0)\LL^2\res\Omega \quad \text{ weakly}^\star \text{ in }\mathcal M(B_1),\end{cases}
\end{equation}
and
$$\int_{B_1} (\eta_{\e_j}+\hat v_{j}^2)|\nabla \hat u_{j}|^2\, dy\to \pi \frac{d\lambda_1}{d\LL^2}(x_0), \quad \frac{1}{\rho_j} \int_{B_1} \left(\hat \e_j |\nabla \hat v_{j}|^2 + \frac{(\hat v_{j}-1)^2}{4\hat \e_j}\right)dy \to \pi \frac{d\lambda_2}{d\LL^2}(x_0).$$
By the classical compactness argument of the Ambrosio-Tortorelli functional (see Lemma \ref{cv faible vgradu}), the previous convergences imply that  $(\eta_{\e_j}+\hat v_j^2)^{1/2} \nabla \hat u_j \rightharpoonup \nabla \hat u=\nabla u(x_0)$ weakly in $L^2(B_1;\R^2)$.
Now using that $(u_\e,v_\e)$ is a critical point of the Ambrosio-Tortorelli functional, we infer that ${\rm div}((\eta_{\e_j}+\hat v_j^2) \nabla \hat u_j)=0$ in $B_1$. Note that this partial differential equation is equivalent to the minimality property
$$\int_{B_1} (\eta_{\e_j}+\hat v_j^2) |\nabla \hat u_j|^2 \, dy \leq \int_{B_1} (\eta_{\e_j}+\hat v_j^2 )|\nabla z|^2 \, dy \quad \text{ for all }z \in \hat u_j + H^1_0(B_1).$$
Let $\varphi \in \mathcal C_c^\infty(B_1)$ be a cut-off function with $0 \leq \varphi \leq 1$ in $\Omega$. Since $\hat u \in H^1(B_1)$, we are allowed to take $z=\varphi \hat u + (1-\varphi)\hat u_j$ as competitor which leads to
$$\int_{B_1} (\eta_{\e_j}+\hat v_j^2) |\nabla \hat u_j|^2 \, dy \leq \int_{B_1} (\eta_{\e_j}+\hat v_j^2)  |\varphi \nabla \hat u+(1-\varphi)\nabla \hat u_j+ \nabla \varphi(\hat u_j-\hat u)|^2 \, dy.$$
Expanding the square and using that $0 \leq \hat v_j \leq 1$, and that  \(0\leq \varphi^2\leq \varphi\), we find that 
\begin{eqnarray*}
\int_{B_1} (\eta_{\e_j}+\hat v_j^2) |\nabla \hat u_j|^2 \, dy & \leq & \int_{B_1} \left( \varphi(\eta_{\e_j}+1)| \nabla \hat u|^2+(1-\varphi)(\eta_{\e_j}+\hat v_j^2) |\nabla \hat u_j|^2\right) dy\\
&&+ \int_{B_1} |\nabla \varphi|^2 (\hat u_j-\hat u)^2 \, dy\\
&&+ 2\int_{B_1}(\eta_{\e_j}+\hat v_j^2)(\hat u_j-\hat u) ( \varphi \nabla \hat u + (1-\varphi) \nabla \hat u_j  )\cdot \nabla \varphi\, dy.
\end{eqnarray*}
Reorganizing the terms in the right-hand side and recalling that $\nabla \hat u=\nabla u(x_0)$, we obtain that
\begin{eqnarray*}
\int_{B_1} (\eta_{\e_j}+\hat v_j^2) |\nabla \hat u_j|^2\varphi \, dy & \leq & (\eta_{\e_j}+1)|\nabla u(x_0)|^2 \int_{B_1}\varphi\, dy+ \int_{B_1} |\nabla \varphi|^2 (\hat u_j-\hat u)^2 \, dy\\
&&+ 2\int_{B_1}(\eta_{\e_j}+\hat v_j^2)(\hat u_j-\hat u) ( \varphi \nabla \hat u + (1-\varphi) \nabla \hat u_j  )\cdot \nabla \varphi\, dy.
\end{eqnarray*}
Since $\hat v_j \to 1$ in $L^2(B_1)$, $\hat u_j \to \hat u$ in $L^2(B_1)$ and $\{(\eta_{\e_j}+\hat v_j^2) \nabla \hat u_j\}_{j \in \mathbb N}$ is bounded in $L^2(B_1;\R^2)$, we can pass to the limit and get that 
$$\limsup_{j \to \infty}\int_{B_1} (\eta_{\e_j}+\hat v_j^2) |\nabla \hat u_j|^2\varphi \, dy \leq |\nabla u(x_0)|^2 \int_{B_1} \varphi\, dy.$$
Combining this inequality with the weak $L^2(B_1;\R^2)$ convergence of  $\{(\eta_{\e_j}+\hat v_j^2)^{1/2} \nabla \hat u_j\}_{j \in \mathbb N}$ to $\nabla \hat u=\nabla u(x_0)$, we infer that
$(\eta_{\e_j}+\hat v_j^2)^{1/2} \nabla \hat u_j\to \nabla u(x_0)$ strongly in $L^2_{\rm loc}(B_1)$. Recalling the last convergence in \eqref{eq:convergences}, we deduce that $\frac{d\lambda_1}{d\LL^2}(x_0)=|\nabla u(x_0)|^2$, hence we obtain \eqref{eq:sigma-singular}.
\end{proof}

\begin{remark}{\rm
If $\nabla u \in L^\infty(\Omega;\R^2)$, we have an alternative proof of the previous result based on Anzellotti's duality pairing (see \cite{A83}). Indeed, recalling that ${\rm div}(\nabla u)=0$ in $\mathcal D'(\Omega)$, we can define as in  \cite[Definition 1.4]{A83} the distribution $[\nabla u \cdot Du]\in \mathcal D'(\Omega)$ by
\begin{equation}\label{eq:Anzellotti}
\langle [\nabla u \cdot Du],\varphi\rangle:=-\int_\Omega u \nabla u \cdot \nabla \varphi \,dx \quad \text{ for all }\varphi \in \mathcal C^\infty_c(\Omega).
\end{equation}
By \cite[Theorem 1.5 and Theorem 2.4]{A83}, $ [\nabla u \cdot Du]$ extends to a bounded Radon measure in $\Omega$ whose absolutely continuous part is given by
\begin{equation}\label{eq:abs-cont}
[\nabla u \cdot Du]^a=|\nabla u|^2\LL^2\res \Omega.
\end{equation}

We now make the connexion between the measures $\lambda_1$ and $[\nabla u \cdot Du]$. Indeed, let $\varphi \in \mathcal C^\infty_c(\Omega)$, taking $u_\e \varphi$ as test function in the first equation of \eqref{criticalpointAT} yields
$$\int_\Omega (\eta_\e +v_\e^2)|\nabla u_\e|^2 \varphi\, dx = - \int_\Omega u_\e (\eta_\e+v_\e^2)\nabla u_\e \cdot \nabla \varphi\, dx.$$
Since $\{u_\e\}_{\e>0}$ and $\{v_\e\}_{\e>0}$ are bounded in $L^\infty(\Omega)$ and $(u_\e,v_\e) \to (u,1)$ in strongly in $[L^2(\Omega)]^2$, we have $u_\e \to u$ and $(\eta_\e+v_\e^2)^{1/2} \to 1$ strongly in $L^4(\Omega)$. Recalling that $(\eta_\e+v_\e^2)^{1/2}\nabla u_\e \wto \nabla u$ weakly in $L^2(\Omega;\R^2)$ and $(\eta_\e +v_\e^2)|\nabla u_\e|^2\LL^2\res \Omega \wto \lambda_1$ weakly$^\star$ in $\mathcal M(\Omega)$, we obtain
\begin{equation}\label{crochet Anzellotti}
\int_\Omega \varphi \, d\lambda_1= -\int_\Omega u \nabla u \cdot \nabla \varphi\, dx.
\end{equation}
In orther words, $\lambda_1=[\nabla u \cdot Du]$, and thus \eqref{eq:abs-cont} implies that $\lambda_1 -|\nabla u|^2\LL^2\res \Omega$ is a singular measure with respect to $\LL^2$.
}
\end{remark}

\begin{remark}\label{rk_2_Anzellotti}\rm Anzellotti's pairing \eqref{eq:Anzellotti} is still well-defined if we assume only \(\nabla u \in L^2(\Omega;\R^2)\) and \({\rm div}(\nabla u)=0\). In this remark we show that this pairing can be used to give a weak formulation to the criticality condition \(\p_{\nu_u} u=0\) on \(J_u\), valid for critical points of the MS energy when \(J_u\) is smooth enough. We recall that this criticality condition comes from \eqref{eq:3rd_eq_MS}. Now, formally we have that \([\nabla u\cdot Du]=\nabla u\cdot Du=|\nabla u|^2\mathcal{L}^2\res \Omega+(u^+-u^-) \nabla u \cdot \nu_u \mathcal{H}^1\res J_u\). Hence a possible weak formulation of \eqref{eq:3rd_eq_MS} could be 
\[ 
\begin{cases}
\dive (\nabla u)=0 \text{ in }\mathcal D'(\Omega)\\
[\nabla u\cdot Du]=|\nabla u|^2\mathcal{L}^2\res \Omega.
\end{cases}\]
Notice that, in light of \eqref{crochet Anzellotti}, the second condition is fulfilled when $(\eta_\e+v_\e^2)^{1/2}\nabla u_\e \to \nabla u$ strongly in $L^2(\Omega)$  which is the case in the context of \cite{BMRa23}.
\end{remark}

Proposition \ref{prop:sigma} allows us to compute the density of $\mu$, the weak\(^\star\) limit of \( (\eta_\e+v_\e^2) \nabla u_\e\otimes \nabla u_\e \mathcal{L}^2\res\Omega\), with respect to $\mathcal{L}^2$.

\begin{proposition}\label{cv xor}
There exists a measure $\mu^{s} \in \mathcal M(\Omega;\mathbb M^{2 \times 2}_{\rm sym})$ which is singular with respect to the Lebesgue measure such that

\begin{equation}\label{eq:mu}
(\eta_\e  + v_\varepsilon^2) \nabla u_\varepsilon \otimes \nabla u_\varepsilon \LL^2\res \Omega \overset{\star}{\rightharpoonup} \mu= \nabla u \otimes \nabla u \LL^2\res\Omega  + \mu^{s}\quad\text{weakly}^\star \text{ in }\mathcal M(\Omega;\mathbb M^{2 \times 2}_{\rm sym}).
\end{equation}
\end{proposition}

\begin{proof}
Let $\sigma$ be defined as in Proposition \ref{prop:sigma}. 
Since the family $\{(v_\varepsilon^2 + \eta_\varepsilon) \nabla u_\varepsilon \otimes \nabla u_\varepsilon\}_{\e>0}$ is bounded in $L^1(\Omega;\mathbb M^{2 \times 2}_{\rm sym})$ thanks to the energy bound \eqref{energy bound}, we infer that, up to a subsequence, $(\eta_\varepsilon + v_\varepsilon^2) \nabla u_\varepsilon \otimes \nabla u_\varepsilon \LL^2\res\Omega \overset{\star}{\rightharpoonup}\mu$ weakly$^\star$ in $\mathcal M(\Omega;\mathbb M^{2 \times 2}_{\rm sym})$ for some $\mu \in \mathcal M(\Omega;\mathbb M^{2 \times 2}_{\rm sym})$. Writing the Lebesgue decomposition of $\mu$ as $\mu=\mu^a + \mu^{s}$, we claim that $\mu^a=\nabla u \otimes \nabla u \LL^2\res\Omega$.

For $i=1$, $2$, we denote by $\sigma_i\in \mathcal M(\Omega)$ the nonnegative measure such that 
\[(v_\varepsilon^2 + \eta_\varepsilon) |\partial_i u_\varepsilon|^2\LL^2\res\Omega \overset{\star}{\rightharpoonup} |\partial_i u|^2\LL^2\res\Omega + \sigma_i \text{ weakly}^\star \text{ in } \mathcal M(\Omega).\] 
Let $\sigma$ be the singular measure provided by Proposition \ref{prop:sigma}. As $\sigma= \sigma_1+ \sigma_2$ with $\sigma_1$ and $\sigma_2$ non negative, we infer that $\sigma_1$ and $\sigma_2$ are absolutely continuous with respect to $\sigma$. Since $\sigma$ is singular with respect to  $\mathcal{L}^2$, we infer that both $\sigma_1$ and  $\sigma_2$ are singular with respect to $\mathcal{L}^2$ as well. 
In addition, since $(\eta_\varepsilon + v_\varepsilon^2)|\partial_i u_\varepsilon |^2 \LL^2\res\Omega  \overset{\star}{\rightharpoonup}  \mu_{ii}^a + \mu^{s}_{ii}$, we obtain, by uniqueness of the Lebesgue decomposition, that $\mu^{s}_{ii} = \sigma_i$ and $\mu^a_{ii} = |\partial_i u |^2\LL^2\res\Omega$. 

\medskip

Let us write $(v_\varepsilon^2+ \eta_\varepsilon)\partial_1 u_\varepsilon \, \partial_2 u_\varepsilon\LL^2\res\Omega \overset{\star}{\rightharpoonup} \partial_1 u \, \partial_2 u\LL^2\res\Omega + \sigma_{12}$ weakly$^\star$ in $\mathcal M(\Omega)$ for some $\sigma_{12} \in \mathcal M(\Omega)$. For every test function $\varphi\in \mathcal C_c(\Omega)$, Proposition \ref{prop:sigma} and Lemma \ref{cv faible vgradu} yield
\begin{align*}
\int_\Omega \left|\sqrt{v_\varepsilon^2+ \eta_\varepsilon} \nabla u_\varepsilon - \nabla u \right|^2 \varphi \, dx &= \int_\Omega (v_\varepsilon^2+ \eta_\varepsilon) |\nabla u_\varepsilon |^2 \varphi\, dx+ \int_\Omega |\nabla u|^2 \varphi \, dx\\
& \quad \quad  - 2 \int_\Omega \sqrt{v_\varepsilon^2+ \eta_\varepsilon} \nabla u_\varepsilon \cdot \nabla u \varphi\, dx \\
&\underset{\varepsilon\rightarrow 0}{\longrightarrow}  \int_\Omega |\nabla u|^2 \varphi\, dx + \int_\Omega \varphi \, d\sigma + \int_\Omega |\nabla u|^2\varphi\, dx - 2 \int_\Omega |\nabla u|^2 \varphi\, dx\\
&= \int_\Omega \varphi\, d\sigma.
\end{align*}
Thus, we have $|\sqrt{v_\varepsilon^2 +\eta_\varepsilon} \nabla u_\varepsilon - \nabla u |^2\LL^2\res  \overset{\star}{\rightharpoonup}  \sigma$ weakly$^\star$ in $\mathcal M(\Omega)$. Hence
\begin{multline}\label{double-produit-mesure}
2\left|(\sqrt{v_\varepsilon^2 +\eta_\varepsilon} \partial_1 u_\varepsilon - \partial_1 u) (\sqrt{v_\varepsilon^2 +\eta_\varepsilon}\partial_2 u_\varepsilon - \partial_2 u)\right|   \\
\leqslant \left(\sqrt{v_\varepsilon^2 +\eta_\varepsilon} \partial_1 u_\varepsilon - \partial_1 u\right)^2 + \left(\sqrt{v_\varepsilon^2 +\eta_\varepsilon} \partial_2 u_\varepsilon - \partial_2 u\right)^2\\
= \left|\sqrt{v_\varepsilon^2 +\eta_\varepsilon}  \nabla u_\varepsilon - \nabla u\right|^2 \overset{\star}{\rightharpoonup} \sigma. 
\end{multline}
Next, expanding the left-hand side of \eqref{double-produit-mesure} and using Lemma \ref{cv faible vgradu}, we obtain for all $\varphi\in \mathcal C_c(\Omega)$,

\begin{multline}\label{eq:weak_con_crossed}
\displaystyle\int_\Omega \left(\sqrt{v_\varepsilon^2 +\eta_\varepsilon}\partial_1 u_\varepsilon - \partial_1 u\right) \left(\sqrt{v_\varepsilon^2 +\eta_\varepsilon} \partial_2 u_\varepsilon - \partial_2 u\right)\: \varphi \, dx \\
=\displaystyle\int_\Omega (v_\varepsilon^2+\eta_\varepsilon )\partial_1 u_\varepsilon\: \partial_2 u_\varepsilon \: \varphi \, dx - \int_\Omega \sqrt{v_\varepsilon^2 +\eta_\varepsilon}\:\partial_1 u_\varepsilon \: \partial_2 u \: \varphi\, dx
- \displaystyle\int_\Omega \sqrt{v_\varepsilon^2 +\eta_\varepsilon}\: \partial_2 u_\varepsilon \: \partial_1 u \varphi\, dx\\
+ \int_\Omega \partial_1 u\: \partial_2 u \varphi \, dx\\
\displaystyle \underset{\varepsilon\rightarrow 0}{\longrightarrow} \displaystyle \int_\Omega \partial_1 u \: \partial_2 u\: \varphi +\int_\Omega \varphi \: d\sigma_{12} - \int_\Omega\partial_1 u \: \partial_2 u\: \varphi\, dx =\int_\Omega \varphi \, d\sigma_{12}. 
\end{multline}
By using \eqref{double-produit-mesure} and \eqref{eq:weak_con_crossed} along with a measure theoreric argument (see e.g.\ in \cite[Lemma 5.1]{Schochet_1995}), we conclude that
$| \sigma_{12} |  \leqslant \frac{1}{2} \sigma$, 
from which we deduce that $\sigma_{12}$ is absolutely continuous with respect to $\sigma$.
Hence, $\sigma_1, \sigma_2$ and $\sigma_{12}$ are singular with respect to $\LL^2$ thanks to Proposition \ref{prop:sigma}.
We have thus established that
$$\mu=\mu^a +\mu^{s}=\nabla u \otimes \nabla u\LL^2\res\Omega+
\begin{pmatrix}
\sigma_1 & \sigma_{12}\\
\sigma_{12} & \sigma_2
\end{pmatrix},
$$
with $\sigma_1$, $\sigma_2$ and $\sigma_{12}$ singular with respect to $\LL^2$.
By uniqueness of the Lebesgue decomposition, we infer that 
$$\mu^a=\nabla u \otimes \nabla u\LL^2\res\Omega, \quad \mu^{s}=\begin{pmatrix}
\sigma_1 & \sigma_{12}\\
\sigma_{12} & \sigma_2
\end{pmatrix}$$
which completes the proof of the proposition.
\end{proof}

\begin{remark}
{\rm
Propositions \ref{prop:sigma} and \ref{cv xor} show that one should not expect the strong $L^2$-convergence of $\{\sqrt{\eta_\e+v_\e^2}\nabla u_\e\}_{\e>0}$ to $\nabla u$. This is an obstacle to be able to pass to the limit in the second inner variation as done in \cite[Theorem 1.3]{BMRa23}.
}
\end{remark}

\section{Proof of Theorem \ref{Theorem 1}}\label{sec:proof of Th1}

In this subsection, we prove Theorem \ref{Theorem 1} by passing to the limit in \eqref{variations internes AT}.
Recall the definition of $\mu$ \eqref{eq:mu}. Observe that $\mu_{12}=\mu_{21}$ and that, passing to the trace, $(\eta_\e+v_\e^2)|\nabla u_\varepsilon |^2 \LL^2 \res\Omega \rightharpoonup \mu_{11} +\mu_{22}$. 
As in \cite[Lemma 5.1]{BMRa23}, we have the following limit conservation law involving $\mu$, the first moment $\overline A$ of $V$, and  the boundary measure $m \in \mathcal{M}(\p \Omega)$ defined in Lemma \ref{boundary term BMR}.
\begin{lemma}\label{variations internes limites}
For all vector field $X\in \mathcal{C}_c^1(\mathbb{R}^2; \mathbb{R}^2)$, we have  
\begin{multline}\label{measure limite + varifold}
\langle (\mu_{11} + \mu_{22}){\rm Id} - 2\mu ,  DX\rangle + \int_{\widehat{J_u}} \overline A : DX \, d\mathcal{H}^1\\
=-\int_{\partial \Omega} (X\cdot \nu) \, dm + \int_{\partial \Omega} |\partial_\tau g|^2 (X\cdot \nu)\, d\mathcal{H}^1 - 2\int_{\partial \Omega} (\nabla u \cdot \nu)(X\cdot \tau) \partial_\tau g \, d\mathcal{H}^1,
\end{multline}
where $\tau$ is a unit tangent vector to $\partial\Omega$ and $\partial_\tau g=\nabla g \cdot \tau$ is the tangential derivative of $g$ on $\partial\Omega$. 
\end{lemma}

\begin{proof}
Let $X\in \mathcal{C}_c^1(\mathbb{R}^2; \mathbb{R}^2)$ be an arbitrary vector field. By \cite[Proposition 4.2]{BMRa23},
\begin{multline}\label{variations internes AT 2}
\int_{\Omega} (2\sqrt{\eta_\varepsilon + v_\varepsilon^2} \nabla u_\varepsilon\otimes \sqrt{\eta_\varepsilon + v_\varepsilon^2} \nabla u_\varepsilon - (\eta_\varepsilon + v_\varepsilon^2)|\nabla u_\varepsilon|^2 {\rm Id} ) : DX \: dx \\
+ \int_\Omega \left[2 \e \nabla v_\varepsilon \otimes \nabla v_\varepsilon -\left(\frac{(1-v_\e)^2}{\varepsilon} + \varepsilon |\nabla v_\e|^2\right){\rm Id} \right] : DX \: dx \\
= \int_{\partial \Omega}\left[(\eta_\varepsilon + 1) |\partial_\nu u_\varepsilon|^2 + \varepsilon |\partial_\nu v_\varepsilon|^2 -(\eta_\varepsilon + 1) |\partial_\tau g|^2 \right] (X\cdot \nu ) d\mathcal{H}^1\\
+ 2(\eta_\varepsilon + 1) \int_{\partial \Omega} \partial_\nu u_\varepsilon (X\cdot \tau) (\partial_\tau g) \, d\mathcal{H}^1,
\end{multline}
We now study separately each term of this expression.

\medskip

Using the definition \eqref{def de mu} of the measure $\mu$, we have

\begin{multline}\label{Lemme I}
\int_{\Omega} (2\sqrt{\eta_\varepsilon + v_\varepsilon^2} \nabla u_\varepsilon\otimes \sqrt{\eta_\varepsilon + v_\varepsilon^2} \nabla u_\varepsilon - (\eta_\varepsilon + v_\varepsilon^2)|\nabla u_\varepsilon|^2 {\rm Id} ) : DX \: dx\\
\xrightarrow[\e \to 0]{}\langle 2\mu-(\mu_{11} + \mu_{22}){\rm Id}  ,  DX\rangle.
\end{multline}

\medskip

Next, thanks to the equi-partition of the energy \eqref{equipartition energy} and the varifold convergence \eqref{cv varifold}, we get 
\begin{equation}\label{Lemme II}
\begin{array}{lll}
&\displaystyle\lim\limits_{\varepsilon \rightarrow 0} \int_\Omega \left[2 \e\nabla v_\varepsilon \otimes \nabla v_\varepsilon -\left(\frac{(1-v_\e)^2}{\varepsilon} + \varepsilon |\nabla v_\e|^2\right)\right] : DX \: dx \\
&\\
=&\displaystyle\lim\limits_{\varepsilon \rightarrow 0} \int_{\{|\nabla v_\varepsilon | \neq 0\} \cap \Omega} 2\varepsilon|\nabla v_\varepsilon|^2 \left(\frac{\nabla v_\varepsilon}{|\nabla v_\varepsilon|} \otimes \frac{\nabla v_\varepsilon}{|\nabla v_\varepsilon|} - {\rm Id} \right) : DX \: dx \\
&\\
=&\displaystyle\lim\limits_{\varepsilon \rightarrow 0} \int_{\{|\nabla w_\varepsilon | \neq 0\} \cap \Omega} |\nabla w_\varepsilon | \left(\frac{\nabla w_\varepsilon}{|\nabla w_\varepsilon|} \otimes \frac{\nabla w_\varepsilon}{|\nabla w_\varepsilon|} - {\rm Id}\right) : DX \: dx\\
&\\
=& -\displaystyle\lim\limits_{\varepsilon \rightarrow 0} \int_{\overline \Omega \times \mathbf G_1} A : DX(x)\, dV_\e(x,A)\\
= & -\displaystyle\int_{\overline \Omega \times \mathbf G_1} A : DX(x)\, dV(x,A)=- \displaystyle\int_{\widehat{J_u}} \overline A : DX d\mathcal{H}^1.\\
\end{array}
\end{equation} 

\medskip

Finally, according to the regularity properties of $\partial \Omega$ and $g$, Lemma \ref{boundary term BMR} yields
\begin{multline}\label{lemme III}
\int_{\partial \Omega}\left[(\eta_\varepsilon + 1) |\partial_\nu u_\varepsilon|^2 + \varepsilon |\partial_\nu v_\varepsilon|^2 -(\eta_\varepsilon + 1) |\partial_\tau g|^2 \right] (X\cdot \nu ) \, d\mathcal{H}^1 \\
\xrightarrow[\e \to 0]{} \int_{\partial \Omega} (X\cdot \nu)\, dm -\int_{\partial \Omega} (X\cdot \nu) |\partial_\nu g|^2 \, d \mathcal{H}^1
\end{multline}
and
\begin{equation}\label{lemme IV}
2(\eta_\varepsilon + 1) \int_{\partial \Omega} (\partial_\nu u_\varepsilon) (X\cdot \tau) \partial_\tau g \, d\mathcal{H}^1 \xrightarrow[\e \to 0]{} 2 \int_{\partial\Omega} (\nabla u \cdot \nu) (X\cdot \tau) \partial_\tau g \, d\mathcal{H}^1. 
\end{equation}

Combining \eqref{Lemme I}, \eqref{Lemme II},\eqref{lemme III}, \eqref{lemme IV} yields \eqref{measure limite + varifold}.  
\end{proof}

Our strategy is then to analyse each term of the left-hand-side of \eqref{measure limite + varifold} to recover \eqref{inner variation of MS}. We now set
\begin{equation}\label{definition de T}
T:= (\mu_{11}+\mu_{22}) \, {\rm Id}  - 2\mu + \overline A \mathcal{H}^1 \res \widehat{J_u} \in \mathcal{M}(\overline{\Omega};\mathbb{M}^{2\times 2}).
\end{equation}
We observe that \eqref{measure limite + varifold} rewrites as
\begin{multline}\label{eq:TDX}
\langle T, DX\rangle =
-\int_{\partial \Omega} (X\cdot \nu) \, dm + \int_{\partial \Omega} |\partial_\tau g|^2 (X\cdot \nu)\, d\mathcal{H}^1\\
 - 2\int_{\partial \Omega} (\nabla u \cdot \nu)(X\cdot \tau) \partial_\tau g \, d\mathcal{H}^1 \quad \text{ for all }X \in \mathcal C^1_c(\R^2;\R^2),
\end{multline}
or still
\begin{equation}\label{eq:divT=mesure}
-{\rm div}(T)= -\nu m \res \partial \Omega + |\partial_\tau g|^2 \nu \HH^1 \res \partial \Omega - 2(\nabla u\cdot \nu)(\partial_\tau g)\tau \HH^1 \res \partial\Omega \quad \text{ in } \mathcal{D}'(\R^2;\R^2).
\end{equation}

Let us write the Lebesgue-Besicovitch decomposition for $\mu$ and $T$, 
\begin{equation}\label{decomposition}
\begin{cases}
\mu = \mu^{a} + \mu^{j} + \mu^{c},\\
T = T^{a} + T^{j} + T^{c},
\end{cases}
\end{equation}
where $\mu^a$ and $T^a$ are absolutely continuous with respect to $\LL^2$, $\mu^j$ and $T^j$ are absolutely continuous with respect to $\HH^1\res \widehat{J_u}$, and $\mu^c$ and $T^c$ are singular with respect to both $\LL^2$ and $\HH^1\res \widehat{J_u}$.
We observe that, according to Proposition \ref{cv xor}, 
\begin{equation}\label{mesure T}
\begin{cases}
T^{a} = (\mu_{11}^{a}+\mu_{22}^{a}){\rm Id}-2\mu^a = (|\nabla u|^2 {\rm Id} - 2\nabla u\otimes \nabla u)\LL^2 , \\
T^j = (\mu_{11}^j+\mu_{22}^j) {\rm Id} - 2\mu^j + \overline A \mathcal{H}^1 \res \widehat{J_u},\\
T^c =  (\mu_{11}^c+\mu_{22}^c) {\rm Id}- 2\mu^c.
\end{cases}
\end{equation}   
Let $\Theta$ be the density of $T^j$ with respect to $\mathcal{H}^1 \res \widehat{J_u}$, so that 
\begin{equation}\label{def Theta j}
T^j = \Theta \mathcal{H}^1 \res \widehat{J_u}=(\mu_{11}^j+\mu_{22}^j) {\rm Id} - 2\mu^j + \overline A \mathcal{H}^1 \res \widehat{J_u}.
\end{equation}

We now establish algebraic properties of the measures $T^j$ and $T^c$. The theory developed in \cite{DePhilippis_Rindler_2016} aims precisely at describing singular measures satisfying a linear PDE. It states that the polar of the singular part of the measure belongs to the wave-cone of the linear operator.

\begin{lemma}\label{dfr}
For $\HH^1$-a.e. $x \in \widehat{J_u}$, the matrix $\Theta(x) \in \mathbb M^{2 \times 2}_{\rm sym}$ is an orthogonal projector. Moreover, $T^c=0$ and there exists a scalar measure $\lambda^c \in \mathcal M(\overline\Omega)$ such that $\mu^c = \lambda^c\, {\rm Id}$.
\end{lemma}

\begin{proof}
According to \eqref{eq:divT=mesure}, the limiting stress-energy tensor is a measure $T$ (extended by zero in \(\R^2\setminus \overline{\Omega}\)) which satisfies the linear PDE ${\rm div}(T) \in \mathcal M(\R^2;\R^2)$. 
The wave cone of the divergence operator (acting on symmetric matrices) is defined by
$$\Lambda:=\bigcup_{\xi \in \R^2, \, |\xi|=1} \{ A \in \mathbb M^{2 \times 2}_{\rm sym} : \; A\xi=0\}$$
and it is immediate to check that it corresponds to the set of singular matrices. Hence according to \cite[Corollary 1.13]{DePhilippis_Rindler_2016} and using that $T^j$ and $T^c$ are singular to each other and both singular with respect to $\LL^2$, one finds that 
\begin{equation}\label{det Tj}
{\rm rank}\left(\Theta\right)\leq 1 \quad \HH^1\text{-a.e. on }\widehat{J_u}
\end{equation}
and
\begin{equation}\label{det Tc}
{\rm rank}\left(\frac{dT^c}{d|T^c|}\right)\leq 1 \quad |T^c|\text{-a.e. in }\overline \Omega.
\end{equation}
Since the matrix-valued measure $(\mu_{11}+\mu_{22}){\rm Id} -2\mu$ has zero trace, it follows from \eqref{mesure T} together with Proposition \ref{premier moment} that
$$\Tr\left(\Theta\right)=\Tr(\overline A)=1 \quad \HH^1\text{-a.e. on }\widehat{J_u}.$$
Hence, recalling \eqref{det Tj}, for $\HH^1$-a.e. $x \in \widehat{J_u}$, the eigenvalues of $\Theta(x)\in \mathbb M^{2 \times 2}_{\rm sym}$ are exactly $0$ and $1$ so that $\Theta(x)$ is indeed an orthogonal projector. 

Using \eqref{mesure T} and \eqref{det Tc}, it follows that for $|T^c|$-a.e. $x \in \overline \Omega$, $\frac{dT^c}{d|T^c|}(x) \in \mathbb M^{2 \times 2}_{\rm sym}$ is a trace free matrix with zero determinant, which implies that $\frac{dT^c}{d|T^c|}(x)=0$. It thus follows that $T^c=0$ and using again \eqref{mesure T},
$$\mu^c=\frac{\mu_{11}^c + \mu_{22}^c}{2} {\rm Id},$$
which completes the proof of the result by setting $\lambda^c:=\frac{\mu_{11}^c + \mu_{22}^c}{2}$.
\end{proof}

Note that, in the previous argument, we crucially used the fact that the dimension is set to be equal to two.

Reporting the information obtained in \eqref{definition de T}, \eqref{mesure T} and Lemma \ref{dfr} inside \eqref{eq:TDX} yields for all $X \in \mathcal C^1_c(\R^2;\R^2)$,
\begin{multline}\label{eq:cons-law}
\int_\Omega (|\nabla u|^2{\rm Id} -2 \nabla u \otimes \nabla u):DX\, dx + \int_{\widehat{J_u}} \Theta  : DX \, d\mathcal{H}^1\\
=-\int_{\partial \Omega} (X\cdot \nu) \, dm + \int_{\partial \Omega} |\partial_\tau g|^2 (X\cdot \nu)\, d\mathcal{H}^1 - 2\int_{\partial \Omega} (\nabla u \cdot \nu)(X \cdot \tau) \partial_\tau g \, d\mathcal{H}^1. 
\end{multline}

We already know that for $\HH^1$-a.e. $x \in \widehat{J_u}$, the matrix $\Theta(x)$ is an orthogonal projector.  We now make this information more precise by showing that it actually consists in the orthogonal projection onto the (approximate) tangent space to the rectifiable set $\widehat{J_u}$ at $x$. We proceed in two steps, by distinguishing interior points in $\Omega$ to boundary points $\partial \Omega$. 

\medskip

We first consider interior points. The proof of the following result is similar to that of \cite[Lemma 5.3]{BMRa23} with several simplications due to the fact that we are here working in dimension two.

\begin{lemma}\label{Theta interieur}
For $\mathcal{H}^1$-a.e $x \in \widehat{J_u} \cap \Omega=J_u$, one has $\Theta(x) = {\rm Id} - \nu_u(x) \otimes \nu_u(x)$. 
\end{lemma}

\begin{proof}
We perform a blow-up argument on \eqref{eq:cons-law}. Let $x_0 \in J_u=\widehat{J_u}\cap \Omega$ be such that

\begin{enumerate}[label=(\roman*)]
\item \label{int i} $J_u$ admits an approximate tangent space at $x_0$, which is given by $T_{x_0} J_u = \nu_u (x_0) ^\perp$. 

\smallskip

\item \label{int ii} $\lim_{\rho \rightarrow 0} \frac{\mathcal{H}^1(J_u \cap B_\rho (x_0))}{2\pi \rho} =1;$

\smallskip

\item \label{int iii} $\lim_{\rho \rightarrow 0} \frac{1}{\rho} \int_{B_\rho (x_0)} |\nabla u|^2\, dx =0;$

\smallskip

\item \label{int iv} $x_0$ is a Lebesgue point of $\Theta$ with respect to $\mathcal{H}^1\res  J_u$. 
\end{enumerate} 

\smallskip

It turns out that $\mathcal{H}^1$-a.e points $x_0$ in $J_u \cap \Omega$ fulfill those four items as a consequence of the Besicovitch Differentiation Theorem (see \cite[Theorem 2.22]{AFP00}), the $\HH^1$-rectifiability of $J_u$ (see \cite[Theorem 2.63]{AFP00}), and the fact the measures $|\nabla u|^2 \mathcal{L}^2 \res \Omega$ and $\mathcal{H}^1 \res J_u$ are singular to each other.

\medskip

Let $\zeta \in \mathcal{C}_c(\mathbb{R}^2; \mathbb{R}^2)$ be a test vector-field such that $\text{Supp}(\zeta) \subset B_1$ and define $\varphi_\delta (x)= \zeta \left(\frac{x-x_0}{\delta}\right)$ for every $\delta >0$ small enough such that $B_\delta (x_0) \subset\subset \Omega$. Note that ${\rm Supp}(\varphi_\delta)  \subset \Omega$ so that $\varphi_\delta$ can be taken as a test function in \eqref{eq:cons-law}. It leads to
\begin{equation}\label{blowup1}
\int_{J_u \cap B_\delta(x_0)} \Theta  : D\varphi_\delta \, d\mathcal{H}^1 = - \int_{B_\delta(x_0)} (|\nabla u|^2 {\rm Id} - 2\nabla u \otimes \nabla u) : D\varphi_\delta \, dx,
\end{equation} 
which rewrites as
\begin{multline}\label{blowup2}
\frac{1}{\delta}\int_{J_u \cap B_\delta(x_0)} \Theta (x) : D\zeta\left(\frac{x-x_0}{\delta}\right) d\mathcal{H}^1(x) \\
= - \frac{1}{\delta}\int_{B_\delta(x_0)} (|\nabla u|^2 {\rm Id} - 2\nabla u \otimes \nabla u) : D\zeta\left(\frac{x-x_0}{\delta}\right)\, dx.
\end{multline}
Since $x_0$ satisfies \autoref{int iii}, one has
\begin{multline}\label{blowup3}
\left|\frac{1}{\delta}\int_{B_\delta(x_0)} (|\nabla u|^2 {\rm Id} - 2\nabla u \otimes \nabla u) : D\zeta\left(\frac{x-x_0}{\delta}\right)\, dx\right|\\
 \leqslant C\|D\zeta\|_{L^\infty(B_1)} \frac{1}{\delta} \int_{B_\delta (x_0)} |\nabla u|^2 \, dx \longrightarrow 0.
\end{multline}
On the other hand, by \autoref{int iv} we have
\begin{multline}\label{blowup4}
\left|\frac{1}{\delta}\int_{J_u\cap B_\delta(x_0)} (\Theta(x) - \Theta(x_0)) : D\zeta\left(\frac{x-x_0}{\delta}\right) \, d \mathcal{H}^1(x) \right| \\
\leqslant C\|D\zeta\|_{L^\infty(B_1)} \frac{1}{\delta} \int_{J_u \cap B_\delta (x_0)}|\Theta (x) - \Theta (x_0)| \, d\mathcal{H}^1(x) \longrightarrow 0.
\end{multline}
Combining \eqref{blowup2}, \eqref{blowup3}, \eqref{blowup4} and \autoref{int ii} yields
$$0 =  \Theta (x_0) : \lim\limits_{\delta\rightarrow 0} \frac{1}{\delta}\int_{J_u \cap B_\delta(x_0)} D\zeta\left(\frac{x-x_0}{\delta}\right) d\mathcal{H}^1(x)= \int_{T_{x_0}J_u \cap B_1} \Theta(x_0) : D\zeta(x) d\mathcal{H}^1(x).$$
Recalling that $\Theta(x_0)$ is an orthonormal projector, there exists a unit vector $e(x_0) \in \R^2$ such that $\Theta(x_0)={\rm Id} - e(x_0)\otimes e(x_0)$. In particular, the measure $W=(\HH^1 \res T_{x_0}J_u) \otimes \delta_{\Theta(x_0)} \in \mathcal{M}(B_1 \times \mathbf G_1)$ is a $1$-varifold and the previous computation shows that $W$ is stationary in $B_1$, i.e., $\delta W=0$ in $B_1$. Hence, according to the monotonicity formula \eqref{eq:monotonicity}, for all $0<r<\rho<1$,
$$\frac{\mathcal{H}^1(T_{x_0}J_u \cap B_\rho)}{\rho} - \frac{\mathcal{H}^1(T_{x_0}J_u \cap B_r)}{r} \\
 = \int_{T_{x_0}J_u \cap (B_{\rho}\setminus B_r)} \frac{|e(x_0)\cdot y|^2}{|y|^3} \, d\mathcal{H}^1 (y).$$
Since the left-hand side of the previous equality vanishes, we obtain
\[\int_{T_{x_0}J_u \cap (B_{\rho}\setminus B_r)} \frac{|e(x_0)\cdot y|^2}{|y|^3} d\mathcal{H}^1 (y) = 0.\]
We deduce that that $y\cdot e(x_0) =0$ for $\mathcal{H}^1$-a.e. $y\in T_{x_0} J_u \cap B_1$, which implies that $T_{x_0} J_u  = e(x_0)^\perp$ and $e(x_0)= \pm \nu_u(x_0)$. 
\end{proof}

We next address the case of boundary points following the ideas of \cite[Lemma 5.4]{BMRa23}. 

\begin{lemma}\label{Theta bord}
For $\mathcal{H}^1$-a.e $x$ in $\widehat{J_u} \cap \partial\Omega$, we have $\Theta (x) = {\rm Id} -\nu(x) \otimes \nu(x)={\rm Id}-\nu_u(x)\otimes \nu_u(x)$, where $\nu$ is the outward unit normal to $\partial \Omega$. 
\end{lemma}

\begin{proof}
Let $x_0 \in \widehat{J_u}\cap \partial \Omega$ be such that :

\smallskip

\begin{enumerate}[label=(\roman*)]
\item \label{bound i} $x_0$ is a Lebesgue point of $\Theta$ with respect to $\mathcal{H}^1 \res \widehat{J_u} \cap \partial \Omega$; 

\smallskip

\item \label{bound ii} $\widehat{J_u}$ admits an approximate tangent space at $x_0$ which is given by $T_{x_0}\widehat{J_u} = \nu_u (x_0)^\perp$;

\smallskip

\item \label{bound iii} $\nu_u (x_0) = \nu (x_0)$;

\smallskip

\item \label{bound iv} $\lim\limits_{\rho \rightarrow 0} \frac{\mathcal{H}^1(\widehat{J_u} \cap B_\rho (x_0))}{2\pi \rho} =1;$

\smallskip

\item \label{bound v} $\lim\limits_{\rho \rightarrow 0} \frac{1}{\rho} \int_{B_\rho (x_0)\cap \Omega} |\nabla u|^2\, dx =0;$

\smallskip

\item \label{bound vi} $m(\{x_0\})=0$.

\end{enumerate} 

Let us first prove that $\mathcal{H}^1$-a.e points $x$ in $\widehat{J_u} \cap \partial\Omega$ fullfil those conditions. The validity of \autoref{bound i} follows from the Besicovitch Differentiation Theorem (see \cite[Theorem 2.22]{AFP00}). Then, the rectifiability of $\widehat{J_u}$, the locality of the approximate tangent space (see \cite[Theorem 2.85 and Remark 2.87]{AFP00}) and the Besicovitch-Mastrand-Mattila Theorem (see \cite[Theorem 2.63 ]{AFP00}) ensure that \autoref{bound ii}, \autoref{bound iii} and \autoref{bound iv} are also satisfied. The validity of \autoref{bound v} is a direct consequence of the fact that the measures $|\nabla u|^2 \mathcal{L}^2\res \Omega$ and $\mathcal{H}^1\res \widehat{J_u}$ are singular to each other. Finally, since $m$ is a finite measure, it follows that its set of atoms is at most countable, hence $\HH^1$-negligible which ensure that \autoref{bound vi} holds.

\medskip

We now prove that $\nu(x_0)$ belongs to the kernel of $\Theta(x_0)$, which, combined with Lemma \ref{dfr} together with \autoref{bound iii}, implies that $\Theta(x_0)= {\rm Id} - \nu_u (x_0) \otimes \nu_u (x_0)={\rm Id}-\nu(x_0)\otimes \nu(x_0)$.

Let $\psi \in \mathcal{C}^1_c (B_1)$ be a scalar test function, and $\zeta \in \mathcal{C}_c^1(\mathbb{R}^2; \mathbb{R}^2)$ a test vector field. For $\delta>0$, set 
$$\varphi_\delta = \psi\left(\dfrac{\cdot-x_0}{\delta}\right) \zeta  \in \mathcal C^1_c(B_\delta(x_0);\R^2),$$
and notice that 
$$D\varphi_\delta = \psi \left(\dfrac{\cdot -x_0}{\delta}\right) D\zeta \ + \zeta \otimes \dfrac{1}{\delta}\nabla \psi \left(\dfrac{\cdot -x_0}{\delta}\right).$$

Using $\varphi_\delta$ as a test function in \eqref{eq:cons-law} yields 
\begin{multline}\label{boundary1}
\int_{B_\delta(x_0) \cap \Omega} (|\nabla u|^2 -2 \nabla u \otimes \nabla u) : D\varphi_\delta \ dx + \int_{\widehat{J_u} \cap B_\delta (x_0)} \Theta : D\varphi_\delta \ d\mathcal{H}^1  \\
=-\int_{\partial \Omega} (\varphi_\delta\cdot \nu) \,dm + \int_{\partial \Omega} |\partial_\tau g|^2 (\varphi_\delta \cdot \nu)\, d\mathcal{H}^1\\
- 2\int_{\partial\Omega \cap B_\delta (x_0)} (\nabla u \cdot \nu)  \partial_\tau g (\tau\cdot \varphi_\delta) \,d \mathcal{H}^1.
\end{multline}

Observe that 
\begin{multline*}
\left|\int_{B_\delta (x_0) \cap \Omega} (|\nabla u|^2 -2 \nabla u \otimes \nabla u) : (D\zeta) \psi\left(\dfrac{\cdot -x_0}{\delta}\right)dx \right| \\
\leqslant C \|\psi\|_{L^\infty(B_1)} \|D\zeta\|_{L^\infty(\R^2)}  \int_{B_\delta (x_0) \cap \Omega} |\nabla u|^2\, dx\xrightarrow[\delta \to 0]{} 0
\end{multline*}
and, thanks to \autoref{bound v}
\begin{multline*}
\left|\int_{B_\delta (x_0) \cap \Omega} (|\nabla u|^2 -2 \nabla u \otimes \nabla u) : \zeta \otimes \dfrac{1}{\delta}\nabla \psi \left(\dfrac{\cdot -x_0}{\delta}\right)dx\right| \\
\leqslant C \|\nabla\psi\|_{L^\infty(B_1)} \|D\zeta\|_{L^\infty(\R^2)} \frac{1}{\delta} \int_{B_\delta (x_0) \cap \Omega} |\nabla u|^2 dx  \xrightarrow[\delta \to 0]{} 0.
\end{multline*}
We then deduce that
\begin{equation}\label{boundaryI}
\int_{B_\delta(x_0) \cap \Omega} (|\nabla u|^2 -2 \nabla u \otimes \nabla u) : D\varphi_\delta \ dx \xrightarrow[\delta \to 0]{} 0.
\end{equation}

\medskip

Next, we have
\begin{multline}\label{boundaryIIa}
\left|\int_{\widehat{J_u}\cap B_\delta(x_0)} \Theta : (D\zeta) \psi\left(\dfrac{\cdot -x_0}{\delta}\right) d \mathcal{H}^1\right|
\leqslant C \|\psi\|_{L^\infty(B_1)} \|D\zeta\|_{L^\infty(\R^2)} \HH^1(\widehat{J_u}\cap B_\delta(x_0)) \xrightarrow[\delta \to 0]{} 0.
\end{multline}
and according to \autoref{bound i}
\begin{multline}\label{boundaryIIb1}
\left|\int_{\widehat{J_u}\cap B_\delta(x_0)} \left(\Theta(x)-\Theta(x_0) \right):\zeta(x) \otimes \dfrac{1}{\delta}\nabla \psi \left(\dfrac{x -x_0}{\delta}\right) d \mathcal{H}^1\right| \\
\leqslant C \|\nabla \psi\|_{L^\infty(B_1)} \|\zeta\|_{L^\infty(\R^2)}  \frac{1}{\delta} \int_{\widehat{J_u}\cap B_\delta(x_0)} \, |\Theta(x) - \Theta(x_0)| \,d \mathcal{H}^1(x) \xrightarrow[\delta \to 0]{} 0.
\end{multline}
Combining \eqref{boundaryIIa}, \eqref{boundaryIIb1}, the continuity of $\zeta$ and \autoref{bound ii} yield
\begin{align}\label{boundaryIIc}
\lim\limits_{\delta \rightarrow 0} \int_{\widehat{J_u} \cap B_\delta (x_0)} \Theta : D\varphi_\delta \, d \mathcal{H}^1& = \lim\limits_{\delta \rightarrow 0} \frac{1}{\delta}\int_{\widehat{J_u}\cap B_\delta(x_0)} \Theta(x) : \zeta (x) \otimes \nabla \psi \left(\dfrac{x -x_0}{\delta}\right) d \mathcal{H}^1(x)\nonumber\\
& = \lim\limits_{\delta \rightarrow 0} \frac{1}{\delta}\int_{\widehat{J_u}\cap B_\delta(x_0)} \Theta(x_0) : \zeta (x) \otimes \nabla \psi \left(\dfrac{x -x_0}{\delta}\right) d \mathcal{H}^1(x)\nonumber\\
& = \lim\limits_{\delta \rightarrow 0} \frac{1}{\delta}\int_{\widehat{J_u}\cap B_\delta(x_0)} \Theta(x_0) : \zeta (x_0) \otimes \nabla \psi \left(\dfrac{x -x_0}{\delta}\right) d \mathcal{H}^1(x)\nonumber\\
& = \int_{T_{x_0} \widehat{J_u} \cap B_1} \Theta(x_0) : \zeta(x_0)\otimes \nabla \psi (y) \,d\mathcal{H}^1(y).
\end{align}

\medskip

Finally \autoref{bound vi} leads to
\begin{multline}\label{boundaryIII}
\left|-\int_{\partial \Omega} (\varphi_\delta\cdot \nu) \,dm + \int_{\partial \Omega} |\partial_\tau g|^2 (\varphi_\delta \cdot \nu)\, d\mathcal{H}^1- 2\int_{\partial\Omega \cap B_\delta (x_0)} (\nabla u \cdot \nu)  \partial_\tau g (\tau\cdot \varphi_\delta) \,d \mathcal{H}^1\right| \\
\leqslant \|\zeta\|_{L^\infty(\R^2)} \|\psi\|_{L^\infty(B_1)} \left[m(B_\delta (x_0)) +\int_{B_\delta (x_0) \cap \partial \Omega} (|\nabla u\cdot \nu| + |\nabla g|^2)\, d\HH^1\right] \xrightarrow[\delta \to 0]{} 0.
\end{multline}

\medskip

Gathering \eqref{boundary1}, \eqref{boundaryI}, \eqref{boundaryIIc}, \eqref{boundaryIII}  leads to 
\begin{equation}\label{boundary2}
\int_{T_{x_0}\widehat{J_u} \cap B_1} \Theta (x_0) : \zeta(x_0) \otimes \nabla \psi \, d\mathcal{H}^1 = 0.
\end{equation}
We now specify \eqref{boundary2} for $\zeta$ such that $\zeta(x_0) = \nu(x_0)$. Denoting by $\tau(x_0)$ an orthonormal vector to $\nu(x_0)$ we decompose $\nabla \psi$ as
$$\nabla \psi = (\partial_{\tau(x_0)} \psi) \tau(x_0) + (\partial_{\nu(x_0)} \psi) \nu (x_0).$$
Since ${\rm Supp}(\psi) \subset B_1$, then  $\int_{T_{x_0}\widehat{J_u} \cap B_1} \partial_{\tau(x_0)} \psi \, d\HH^1=0$ and \eqref{boundary2} becomes
$$\int_{T_{x_0}\widehat{J_u} \cap B_1} \Theta(x_0):(\nu(x_0) \otimes \nu(x_0) )\partial_{\nu(x_0)} \psi \, d \mathcal{H}^1 =0\quad \text{ for all } \psi \in \mathcal{C}^1_c (B_1).$$
We conclude that 
$$\nu(x_0) \cdot (\Theta(x_0) \nu(x_0)) = \Theta(x_0) : (\nu(x_0) \otimes \nu(x_0))= 0$$
and thus, since \(\Theta(x_0)\) is a projector, we obtain $\Theta(x_0) \nu(x_0)=0$, as claimed. 
\end{proof}

We are now in position to complete the proof of our main result.

\begin{proof}[Proof of \autoref{Theorem 1}]
According to \eqref{eq:cons-law}, Lemma \ref{Theta interieur} and Lemma \ref{Theta bord}, we have that for all $X \in \mathcal C^1_c(\R^2;\R^2)$,
\begin{multline*}
\int_{\Omega} (|\nabla u|^2 -2 \nabla u \otimes \nabla u) : DX \, dx + \int_{\widehat{J_u}} ({\rm Id}-\nu_u \otimes \nu_u) : DX \, d\mathcal{H}^1  \\
=-\int_{\partial \Omega} (X\cdot \nu) \,dm + \int_{\partial \Omega} |\partial_\tau g|^2 (X \cdot \nu)\, d\mathcal{H}^1- 2\int_{\partial\Omega} (\nabla u \cdot \nu) (X\cdot \tau) \partial_\tau g  \,d \mathcal{H}^1.
\end{multline*}
Specifying to vector fields $X \in \mathcal{C}_c^1(\R^2;\R^2)\) such that $X\cdot \nu = 0$ on \(\p \Omega\) leads to
\begin{multline*}
\int_{\Omega} (|\nabla u|^2 -2 \nabla u \otimes \nabla u) : DX \, dx + \int_{\widehat{J_u}} ({\rm Id}-\nu_u \otimes \nu_u) : DX \, d\mathcal{H}^1
=- 2\int_{\partial\Omega} (\nabla u \cdot \nu) X\cdot \nabla g \,d \mathcal{H}^1,
\end{multline*}
which completes the proof of Theorem \ref{Theorem 1}. 
\end{proof}

\begin{remark}{\rm The assumption of convergence of the energy can also be used to pass to the limit in the second inner variation in general, cf.\ \cite{Le_2011,Le_2015,LS19,BMRa23}. Here the sole convergence of the phase field energy is not sufficient a priori to pass to the limit in the second inner variation of the AT energy. 

Thanks to the varifold convergence and the equi-partition of energy, we first observe as in \cite[Corollary 5.1]{BMRa23} that if \(\{(u_\e,v_\e)\}_{\e>0}\) is a family of 
critical points of \(AT_\e\) satisfying the assumptions of Theorem \ref{Theorem 1} then, up to a subsequence
$$\nabla w_\e \otimes \nabla w_\e \LL^2\res \Omega \wto ({\rm Id}-\overline A )\HH^1\res J_u \quad \text{ weakly* in }\mathcal M(\Omega;\mathbb M^{2 \times 2}_{\rm sym})$$
and
$$\e \nabla v_\e \otimes \nabla v_\e \LL^2\res \Omega \wto \frac12 ({\rm Id}-\overline A )\HH^1\res J_u \quad \text{ weakly* in }\mathcal M(\Omega;\mathbb M^{2 \times 2}_{\rm sym}).$$
Recalling that $({\rm Id} - \overline A)\HH^1\res J_u=\nu_u \otimes \nu_u \HH^1\res J_u -((\Tr(\mu^j)){\rm Id} -2\mu^j)$

Using Propositions \ref{cv dirichlet} and \ref{cv xor}, the expression of the second inner variation of the AT energy 
computed in \cite[Lemma A.3]{BMRa23}, we find that if \(X\in \C^\infty_c(\Omega;\R^2)\),
 \begin{eqnarray*}
\lim_{\e \to 0} \delta^2AT_\e(u_\e,v_\e)[X] & = &\int_\Omega |\nabla u|^2 (({\rm div} X)^2 -\text{Tr}((DX)^2))\, dx+ \langle \text{Tr}(\mu^s),({\rm div} X)^2 -\Tr((DX)^2)\rangle\\
&&-4\int_\Omega ((\nabla u \otimes \nabla u):DX) \dive X \, dx -4\langle \mu^s, (\dive X) DX \rangle \\
&&+4\int_\Omega  (\nabla u \otimes \nabla u) :(DX)^2 \, dx +4\langle \mu^s, (DX)^2\rangle \\
&&+2 \int_\Omega |DX^T\nabla u|^2\, dx + 2\langle \mu^s, (DX) (DX)^T \rangle\\
%&&+\int_{J_u} ({\rm Id} -\nu_u \otimes \nu_u):DY\, d\HH^1+\langle (\Tr(\mu^j)){\rm Id} -2\mu^j,DY\rangle\\
&&+\int_{J_u} ((\dive X)^2 -\Tr((DX)^2))\, d\HH^1\\
&&-2\int_{J_u} (\nu_u \otimes \nu_u):(DX) \dive X\, d\HH^1 +2 \langle (\Tr(\mu^j)){\rm Id} -2\mu^j,(DX)\dive X\rangle\\
&&+2\int_{J_u} (\nu_u \otimes \nu_u):(DX) ^2\, d\HH^1 -2 \langle (\Tr(\mu^j)){\rm Id} -2\mu^j,(DX)^2\rangle\\
&&+\int_{J_u} |DX^T \nu_u|^2\, d\HH^1 -\langle (\Tr(\mu^j)){\rm Id} -2\mu^j,(DX)(DX)^T\rangle.
\end{eqnarray*}
Thus it does not seem to have any compensation phenomenon to get rid-off the terms involving the singular measure $\mu$, and stability of the the second inner variation might not be satisfied.}
\end{remark}

\section*{Acknowledgments}
We would like to warmly thank Antonin Chambolle and Fabrice Bethuel for helpful discussions about this paper. This work was supported by a public grant from the Fondation Mathématique Jacques Hadamard.

\bibliographystyle{abbrv}
\bibliography{bibliography}

\end{document}